\documentclass[10pt]{article}
\usepackage{amsfonts,amsmath,latexsym}
\usepackage[T1]{fontenc}
\usepackage[utf8]{inputenc}
\usepackage{hyperref}
\usepackage[dvips]{graphicx}
\usepackage{epsfig}
\usepackage{enumerate}
\usepackage{epsf}
\usepackage{amsthm}
\usepackage{color}
\usepackage{amssymb}
\usepackage{fancyhdr} 

\usepackage{palatino}

\newtheorem{thm}{Theorem}[section]
\newtheorem{propo}[thm]{Proposition}
\newtheorem{lemme}[thm]{Lemma}
\newtheorem{corro}[thm]{Corollary}
\newtheorem{defi}[thm]{Definition}
\newtheorem{remarque}[thm]{Remark}
\newtheorem{example}[thm]{Example}

\def \D{\mathbb D}
\def\R{\mathbb R}
\def\N{\mathbb N}

\def\C{\mathbb C}
\def\E{\mathbb E}

\def\sha{{\cal A}}

\def\shd{{\cal D}}

\def\shf{{\cal F}}

\def\shm{{\cal M}}

\def\shl{{\cal L}}

\def\shr{{\cal R}}

\def\shu{{\cal U}}

\author{
{\sc St\'ephane GOUTTE}
\thanks{CNRS and Laboratoire de Probabilit\'es et Mod\`eles Al\'eatoires (LPMA)
 UMR 7599. Universit\'e Paris 7 Diderot. E-mail:{ \tt stephane.goutte@univ-paris-diderot.fr}} 
\thanks{supported by the FUI project $R=MC^2$.}
 {\sc,}\ {\sc Nadia OUDJANE}
\thanks{EDF R\&D, Universit\'e Paris 13, LAGA. Institut Galil\'ee.  and FiME (Laboratoire de Finance des March\'es de l'Energie
(Dauphine, CREST,  EDF R\&D) www.fime-lab.org)
E-mail:{\tt  
nadia.oudjane@edf.fr}}
\ {\sc and}\ {\sc Francesco RUSSO} 
\thanks{ENSTA-ParisTech. Unit\'e de Math\'ematiques appliqu\'ees.
  E-mail:{\tt  francesco.russo@ensta-paristech.fr}}.
}

\date{February 2nd  2012}
\title{On some expectation and derivative operators
related to integral representations of random variables
with respect to a PII process.}
\newcommand{\MBFigure}[6]{
$\left. \right.$ \\
\refstepcounter{figure}
\addcontentsline{lof}{figure}{\numberline{\thefigure}{\ignorespaces #5}}
\begin{center}
\begin{minipage}{#1cm}
\centerline{\includegraphics[width=#2cm,angle=#3]{#4}}
\begin{center}
\upshape{F\textsc{ig} \normal
\end{center}
size{\thefigure}. $-$} #5
\end{center}
\label{#6}
\end{minipage}
\end{center}
$\left. \right.$ \\}

\begin{document}
\maketitle 
\begin{abstract} Given a process with independent increments $X$ (not necessarily a martingale)
and a large class of square integrable r.v. $H=f(X_T)$, $f$ being
the Fourier transform of a finite measure $\mu$, we provide 
explicit Kunita-Watanabe and F\"ollmer-Schweizer 
decompositions. 
The representation is expressed by means of two significant
maps: the expectation and derivative operators related
to the characteristics of $X$.
We also provide an explicit expression for the variance
 optimal error when hedging the claim $H$ with underlying process $X$. 
Those questions are motivated by finding
the solution  of the celebrated problem of global and local quadratic risk
minimization in mathematical finance.
\end{abstract}

\medskip\noindent  {\bf 2010  AMS-classification}: 60G51, 60H05, 60J75, 91G10
\section{Introduction}

\setcounter{equation}{0}

Let $X$ be an $(\shf_t)$-special cadlag semimartingale, where 
 $(\shf_t)$ is a filtration fulfilling the usual conditions.
It admits a unique decomposition $M + A$
where $M$ is an $(\shf_t)$-local martingale and $A$ is an
$(\shf_t)$-predictable process with bounded variation. 
Given $T > 0$ and a square integrable
random variable $H$ which is $\shf_T$-measurable, we consider three specific issues of stochastic analysis that are particularly relevant in stochastic finance. 
\begin{description}
	\item[Kunita-Watanabe (KW) decomposition.]
	This problem consists in providing existence conditions and explicit expressions of a predictable process  $(Z_t)_{t \in [0,T]}$
and 
an $\shf_0$-measurable r.v.  such that
\begin{equation}
\label{eq:KWIntro}
H = V_0 + \int_0^T Z_s dM_s + O_T \ ,
\end{equation}
where $(O_t)$ is an $(\shf_t)$-local martingale 
such that $\langle O,M \rangle = 0$.\\
When $X=M$ is a classical Brownian motion $W$
  and $(\shf_t)$ is the associated canonical filtration,
$Z$ is provided by the celebrated {\it Clark-Ocone formula}
at least when $H$ belongs to the Malliavin-Sobolev 
type space $\D^{1,2}$.
In that case one has
\begin{equation} 
\label{EClarkOcone}
H = \E(H) + \int_0^T \E(D_s H \vert \shf_s) dW_s ,
\end{equation}
where $D H = (D_t H)_{t \in [0,T]}$ is the classical Malliavin 
derivative of $H$.

In the last ten years a significant scientific production appeared
at the level of Malliavin calculus in relation with Poisson measures in several directions.
A trend which was particularly directed to obtaining a generalization of
Clark-Ocone formula was started by \cite{schoutens00}. In Theorem 1, the
authors obtained a chaos type decomposition of a square integrable random variable $H$ in the
Poisson space generated by a finite number of L\'evy square integrable martingales $(\eta_j)$, with respect to a
well-chosen  sequence of strongly orthogonal martingales $\gamma^{(m)}$.
This could allow to represent any $H$ as an infinite sum of stochastic integrals with respect
to the  $\gamma^{(m)}$, an infinite dimensional {\it derivative} $ {\shd}^{(m)}$ with respect to 
 $\gamma^{(m)}$ and a Malliavin-Sobolev type space $\D^{1,2}$.
A first formulation of a Clark-Ocone type formula was produced by \cite{leon02}: it consisted in representing 
square integrable random variables $H$ with respect to the $\gamma^{(m)}$ in terms of some predictable projections
of  $ {\shd}^{(m)} H$. Another class of stochastic derivative (this time) with respect to $\eta_j$
was introduced by \cite{nunno02}.
With the help of an isometry obtained in \cite{lokka}, one could deduce 
the more intrinsic (and recently widely used) Clark-Ocone type formula of the type
$$ H = \E(H) + \int_0^T \int_\R \E(D_{t,x} H \vert \shf_t) \tilde N(dt,dx)$$
where $\tilde N$ is the compensated Poisson random measure and
$(D_{t,x})$ a two-indexed derivative operator.
This formula is also stated in Theorem 12.16 of \cite{BookDiNunno}.
Theorem 4.1 of \cite{nunno} allows to provide an explicit representation
 of the process  $Z$
appearing in~\eqref{eq:KWIntro} with the help of previous operator $ D_{t,x}$.

	\item[F\"ollmer-Schweizer decomposition.] 
That decomposition is a generalization of the Kunita-Watanabe one
in the sense that square integrable random variables are represented 
with respect to $X$ instead of $M$. It consists in providing existence conditions and explicit expressions of a predictable process $\xi$ and 
an $\shf_0$-measurable square integrable r.v.  $H_0$ 
 such that
\begin{equation} 
\label{BSD}
H = H_0 + \int_0^T \xi_s dX_s + L_T
\end{equation}
 where 
$L_T$ is the terminal value of an orthogonal martingale $L$ to $M$,
 the martingale part of $X$.

In the seminal paper  \cite{FS91},  the problem is treated 
for an  underlying process  $X$  with continuous paths.
In the general case, $X$ is said to satisfy the {\bf structure condition}
(SC)  if there is a predictable process $\alpha$ such that 
$A_t = \int_0^t \alpha_s d\langle M\rangle_s, t \in [0,T]$, and 
$\int_0^T \alpha_s^2 d\langle M\rangle_s < \infty $ a.s.
An interesting connection with the theory of backward stochastic differential
equations (BSDEs)  in the sense of \cite{pardouxpeng}, was proposed in 
\cite{S94}.  \cite{pardouxpeng} considered BSDEs driven by Brownian motion;
in \cite{S94} the Brownian motion is in fact replaced by $M$.
The first author who considered
a BSDE driven by a martingale was \cite{buckdahn}.
 The BSDE problem consists
in finding a triple $(V,\xi, L)$ where
$$ V_t = H - \int_t^T \xi_s dM_s - \int_t^T \xi_s \alpha_s d \langle M\rangle_s
- (L_T - L_t), $$
and $L$ is an $(\shf_t)$-local martingale
orthogonal to $M$. 
The solution  $(V_0, \xi, L)$ of that BSDE constitutes a triplet 
 $(H_0, \xi, L)$ solving \eqref{BSD}.
The FS decomposition is motivated in mathematical finance
by looking for the solution of the so called {\it local risk minimization},
see \cite{FS91} where $H$ represents a contingent claim to hedge and
$X$ is related to the price of the underlying asset.
In this case, $V_t$ represents the {\it hedging portfolio value}
 of the contingent claim
at time $t$, $\xi$ represents the hedging strategy and the initial capital $V_0$ constitutes in fact
 the expectation of $H$ under
the so called {\it minimal martingale measure}, see \cite{S95}.

	\item[Variance optimal hedging.] 
	This approach developed by M. Schweizer~(\cite{S94},~\cite{S95bis}) consists in  minimizing the quadratic distance
 between the hedging portfolio and the pay-off. More precisely,
 it consists in providing existence conditions and explicit expressions 
of a predictable process
  $(\varphi_t)_{t \in [0,T]}$
and 
an $\shf_0$-measurable square integrable r.v.  $V_0$
such that
\begin{equation}
\label{eq:VOIntro}
 (V_0, \varphi)=\textrm{Argmin}_{c,v}\ \E \big (\varepsilon(c,v)\big )^2 \,,\quad \textrm{where}\quad  \varepsilon(c,v)=H-c- \int_0^T v_s dX_s \ .
\end{equation}
The quantity $V_0$  and process  $\varphi$ represent 
 the initial capital and the optimal hedging strategy
of the contingent claim $H$.
\end{description}

When the market is complete and without arbitrage opportunities, the 
representation property \eqref{BSD} holds with $ L \equiv 0$; 
so those three decompositions (Kunita-Watanabe, F\"ollmer-Schweizer and Variance Optimal) reduce to a single representation of the random variable $H$ as a stochastic integral modulo a martingale (\textit{risk neutral}) change of measure. 
If the market model is incomplete (e.g. because of jumps or stochastic volatility in prices dynamics) then those three decompositions are in general  different and  a residual term must be added to each integral representation, e.g. $O_T$ and $L_T$ and $\varepsilon(V_0,\varphi)$. 
However, even in this incomplete market setting, a nice exception occurs if the underlying price $X$ is a martingale. Indeed, the martingale property allows  to bypass some theoretical difficulties leading again to three identical decompositions. 

Most of the articles providing quasi-explicit expressions for those
 decompositions are precisely assuming the martingale property  for the process $X$, therefore coming down to consider the Kunita-Watanabe decomposition. 
For instance in~\cite{JS00}, the authors developed an original approach to find an explicit expression for the Kunita-Watanabe decomposition of a random variable $H$ of the form $H=f(Y_T)$ where $Y$ is a reference Markov process and the price process $X$ is a martingale related to $Y$. 
Their idea is to apply Ito's formula to derive the Doob-Meyer decomposition of $\E[H \vert \mathcal{F}_t]$ and then to write the orthogonality condition between $\E[H \vert  \mathcal{F}_{\cdot}]-\int_0^{\cdot}Z_s dX_s$ and $X$. 
In~\cite{ContTankov}, the authors follow the same idea to derive the hedging
strategy minimizing the \textit{Variance Optimal hedging error} under
the (risk-neutral) pricing measure. They provide some interesting financial motivations for this martingale framework. 
Their approach also applies to a broad class of price models and to some path dependent random variables $H$. In some specific cases they obtain quasi-explicit expressions for the Variance Optimal strategy. For instance, they prove that if $X$ is the exponential of a L\'evy process, then the strategy is related to derivatives and integrals w.r.t. the L\'evy
measure of the conditional expectation $\E[H \vert \mathcal{F}_t]$. 

Unfortunately,  minimizing the quadratic hedging error 
under the pricing measure, 
 can lead to a huge quadratic error under the
objective measure. 
Moreover, the use of Ito's lemma in those approaches requires some regularity conditions on the conditional expectation $\E[H \vert \mathcal{F}_t]$: basically
  it should be once differentiable w.r.t. the time variable and twice differentiable w.r.t. the space variable with continuous partial derivatives. 

In the non-martingale framework, 
one major contribution is due to~\cite{Ka06} whose authors 
 restricted their analysis to the specific case where $X$ is the exponential of a L\'evy process  and
 $H = f(X_T)$,  $f$ being  the Fourier-Laplace transform of a complex finite measure. 
The authors obtained an explicit expression for the process $\xi$ intervening in 
\eqref{BSD}. This result was generalized to exponential of  non stationary processes in the continuous and discrete time setting in~\cite{GOR} 
and~\cite{GORDiscr}. 

Following this approach, the objective of the present paper is to consider the non-martingale framework and to provide quasi explicit expressions of both the Kunita-Watanabe and F\"ollmer-Schweizer decompositions when $X$ is a general process with independent increments 
and $H = f(X_T)$ is the Fourier transform of a finite measure $\mu$. Our method does not rely on Ito's formula and therefore does not require any further regularity condition on conditional expectations. 
The representation is carried by means of two significant
maps: the so-called {\it expectation and derivative operators}
 related to the characteristics of the underlying process $X$.
We also express explicitly the Variance Optimal hedging strategy and the corresponding Variance Optimal error.

The paper is organized as follows. In Section \ref{Generalities} we recall
some essential considerations related to the F\"ollmer-Schweizer decomposition
related to general special semimartingales.  In Section \ref{SPII}
we provide the framework related to processes with independent increments
and related structure conditions.   Section \ref{sec:FSPII} provides the
explicit Kunita-Watanabe and the F\"ollmer-Schweizer decompositions
under minimal assumptions. 
Section \ref{sec4} formulates the solution of the global minimization problem evaluating
the variance of the hedging error.
Finally, in Section \ref{ESGM}, we consider a class of examples, for which we verify that
the assumptions are fulfilled.

\section{Generalities on F\"ollmer-Schweizer decomposition and mean variance hedging}
\label{Generalities}

\setcounter{equation}{0}

In the whole paper, $T>0$, will be a fixed terminal time and we will 
denote by $(\Omega,\mathcal{F},(\mathcal{F}_t)_{t \in [0,T]},P)$ a
 filtered probability space, fulfilling the usual conditions.
We suppose from now on $\shf_0$ to be trivial for simplicity.

\subsection{Optimality and F\"ollmer-Schweizer Structure Condition}
\label{sec:FMStruc}
Let $X=(X_t)_{t \in [0,T]}$ be a real-valued
 special semimartingale with canonical decomposition, $X= M+A$. For the clarity of the reader, we formulate in dimension one,
 the concepts appearing in the literature, see e.g. \cite{S94}
 in the multidimensional case.
In the sequel $\Theta$ will denote the space $L^2(M)$ of all predictable $\mathbb{R}$-valued processes $v=(v_t)_{t \in [0,T]}$ such that 
$\mathbb{E}\left[\int_0^T|v_s|^2d\left\langle M\right\rangle_s\right]<\infty\ $.
For such $v$, clearly $\int_0^t v dX, t \in [0,T]$ is well-defined;
 we denote by $G_T(\Theta)$,
 the space generated by all 
the r.v. $G_T(v)=\int_0^Tv_sdX_s$ 
 with $v=(v_t)_{t \in [0,T]}\in \Theta$.
\begin{defi}\label{Defproblem2}
The \textbf{minimization problem} we aim to study is the following: Given $H \in \mathcal{L}^2$, an admissible strategy pair $(V_0,\varphi)$ will be called \textbf{optimal} if $(c,v)=(V_0,\varphi)$ minimizes the expected squared hedging error
\begin{equation}\mathbb{E}[(H-c-G_T(v))^2]\ ,\label{problem2}\end{equation}
over all admissible strategy pairs $(c,v) \in \mathbb{R} \times \Theta$.
 $V_0$ will represent the {\bf initial capital} of the hedging portfolio 
for the contingent claim $H$ at time zero.
\end{defi}
The definition below introduces an important technical condition, see \cite{S94}.



\begin{defi}\label{defSC}
Let $X=(X_t)_{t \in [0,T]}$ be a real-valued special semimartingale. 
$X$ is said to satisfy the
 \textbf{structure condition (SC)} if there is a predictable
 $\mathbb{R}$-valued process $\alpha=(\alpha_t)_{t \in [0,T]}$ such that 
the following properties are verified.
\begin{enumerate}
	\item  $A_t=\int_0^t \alpha_s d\left\langle M\right\rangle_s\
 ,\quad \textrm{for all}\ t \in [0,T],$ so that
  $dA\ll d\left\langle M\right\rangle $. 
\item  ${\displaystyle \int_0^T \alpha^2_s d\left\langle 
M\right\rangle_s<\infty\ ,\quad P-}$a.s.
\end{enumerate}

\end{defi}
\begin{defi}\label{D29}
From now on, we will denote by $K=(K_t)_{t \in [0,T]}$ the cadlag process $K_t=\int_0^t \alpha^2_s d\left\langle M\right\rangle_s\ ,\quad \textrm{for all}\ t\in [0,T]\ .$ This process will be called the \textbf{mean-variance trade-off} (MVT)  process.
\end{defi}
In~\cite{S94}, the process $(K_t)_{t \in [0,T]}$ is denoted by $(\widehat{K}_t)_{t \in [0,T]}$.
\subsection{F\"ollmer-Schweizer decomposition 
and variance optimal hedging}
\label{sec:FMDecomp}
Throughout this section, as in Section~\ref{sec:FMStruc}, $X$ is supposed to be an $(\mathcal{F}_t)$-special semimartingale fulfilling the (SC) condition.
\begin{defi}\label{DefFSDecomp}
We say that a random variable $H \in \mathcal{L}^2(\Omega,\mathcal{F},P)$ admits a \textbf{F\"ollmer-Schweizer (FS) decomposition}, if 
it  can be written as
\begin{eqnarray}\label{FSdecompo}H=H_0+\int_0^T\xi_s^HdX_s+L_T^H\ ,
 \quad P-a.s.\ ,\end{eqnarray}
where $H_0\in \mathbb{R}$ is a constant, $\xi^H \in \Theta$ and $L^H=(L^H_t)_{t \in [0,T]}$ is a square integrable martingale,  
with $\mathbb{E}[L_0^H]=0$ and  strongly orthogonal to $M$, i.e. $\langle L^H,M\rangle=0$.
\end{defi}
 The notion of strong orthogonality is treated for instance in Chapter~IV.3 p.~179 of~\cite{Pr92}.
%
%
%
%
%
%
%
\begin{thm}\label{ThmExistenceFS}
If X satisfies (SC) and  the MVT process $K$ is uniformly bounded
 in $t$ and $\omega$,
then we have the following.
\begin{enumerate}
\item Every random variable $H \in \mathcal{L}^2
(\Omega,\mathcal{F},P)$ admits a unique FS decomposition.
 Moreover, $H_0 \in \R$, $\xi \in \Theta$ and $L^H$ is uniquely
 determined by $H$.
 \item For every $H \in \mathcal{L}^2(\Omega,\mathcal{F},\mathcal{P})$
  there exists a unique $(c^{(H)},\varphi^{(H)}) \in \R \times \Theta$
 such that
\begin{eqnarray}\label{thm128}
\mathbb{E}[(H-c^{(H)}-G_T(\varphi^{(H)}))^2]=\min_{(c,v)\in 
\R \times \Theta}
\mathbb{E}[(H-c-G_T(v))^2]\ .
\end{eqnarray}
 \end{enumerate}
\end{thm}
%
%
From the F\"ollmer-Schweizer decomposition follows the solution to
 the global minimization problem~(\ref{problem2}).
Next theorem gives the explicit form of the optimal strategy.
\begin{thm}\label{ThmSolutionPb1}
Suppose that X satisfies (SC), that the MVT process $K$ of X is
deterministic and  $\langle M\rangle$ is continuous. Let
  $\alpha$ be the process appearing in 
Definition~\ref{defSC} of~(SC) and  let $H\in \mathcal{L}^2$.
\begin{eqnarray*}
\min_{(c,v)\in 
\R \times \Theta}\mathbb{E}[(H-c-G_T(v))^2]&=&\exp(-K_T)\mathbb{E}[(L_0^H)^2]+\mathbb{E}\left[\int_0^T\exp\{-(K_T-K_s)\}d\left\langle L^H\right\rangle_s\right]\ .\label{corro9Sc2}\end{eqnarray*}
\end{thm}
\begin{proof} \
The result follows from Corollary~9 of~\cite{S94}. We remark that
 being $\langle M\rangle$  continuous, the Dol\'eans-Dade 
exponential of $K$, $\mathcal{E}(X)$, equals $\exp(K)$.
\end{proof}
In the sequel, we will find an explicit expression of the KW and 
 FS decomposition for a large class of square integrable random variables
$H$, when the underlying process is a process with independent increments.
\section{Processes with independent increments (PII)}
\label{SPII}
\setcounter{equation}{0}
This section deals with the case of processes with independent
increments.
First, we recall some useful properties of such
processes, then, we obtain a sufficient condition  on the characteristic
function for the existence of the FS decomposition. \\
Beyond its own theoretical interest, this work is motivated by its
 possible application to hedging derivatives related to financial
or commodity assets. 
Indeed, in some specific cases it is reasonable to introduce arithmetic models 
(eg. Bachelier) in contrast to geometric models (eg. Black-Scholes model), see for instance ~\cite{Benth-Kallsen}.
\subsection{Generalities on PII processes}
Let $X = (X_t)_{t \in [0,T]}$ be a stochastic process. Let $t \in [0,T]$.
\begin{defi}
\label{defPhi}
\begin{enumerate}
\item The {\bf characteristic function} of (the law of) $X_t$ 
is the continuous function 
$$
\varphi_{t}:\mathbb{R} \rightarrow \mathbb{C} \quad \textrm{with}\quad \varphi_{t}(u)=\mathbb{E}[e^{iu X_t}]\ .
$$
\item The {\bf Log-characteristic function} of (the law of) $X_t$ 
is the unique function $\Psi_t: \R \rightarrow \R$ such that
$\varphi_{t} = \exp(\Psi_t(u))$ and  $\Psi_t(0) = 0$.
\end{enumerate}
\end{defi}
Notice that for  $u\in \R$ we have $\overline{\Psi_t(u)}=\Psi_t(-u)$.
Since $\varphi: [0,T] \times \R   \rightarrow \C,$
is uniformly continuous and
$\varphi_t(0) = 1$, then there is a neighborhood $\shu$ of $0$ 
such that 
\begin{equation} \label{BBB}
 {\rm Re} \Psi_t(u) > 0, \forall t \in [0,T], u \in \shu. 
\end{equation}
\begin{defi}\label{defPAI}
$X=(X_t)_{t \in [0,T]}$ is a (real) 
\textbf{process with independent increments (PII)} iff
\begin{enumerate}
\item $X$ 
 has cadlag paths; 
\item $X_0=0$;
\item $X_t-X_s$ is independent of $\mathcal{F}_s$ for $0 \leq s < t \leq
  T$ where $(\mathcal{F}_t)$ is the canonical filtration associated
  with $X$;
\\
moreover we will also suppose
\item $X$ is continuous in probability, i.e. 
 $X$ has no fixed time of discontinuities.
\end{enumerate}
\end{defi}
The process $X$ is said to be  \textbf{square integrable}
if for every  $t \in [0,T]$, $\E[|X_t|^2]<\infty$.

From now on, $(\shf_t)$ will always be the canonical filtration 
associated with $X$. Below, we state some elementary properties
 of the characteristic functions related to PII processes.
In the sequel, we will always suppose that $X$ is a semimartingale.
For more details about those processes the reader can consult
Chapter II of \cite{JS03}.
\begin{remarque}\label{RemarkRCher}
Let $0\leq s < t \leq T$, $u\in \R$,
\begin{enumerate}
\item $\Psi_{X_t}(u)=\Psi_{X_s}(u)+\Psi_{X_t-X_s}(u)$.
\item $\exp\left(iuX_t-\Psi_t(u)\right)$ is an ($\shf_t$)-martingale.
\item There is an increasing function $a: [0,T]\rightarrow \R$ and a triplet 
 $(b_t,c_t,F_t)$ called characteristics such that 
\begin{equation} 
\label{eq:LevyKhinAbs} 
\Psi_t(u)=\int_0^t \eta_s(u)\,da_s\ ,\quad \textrm{for all}\ u\in\mathbb R \ .
\end{equation}
where $\eta_s(u):=\left [iu b_s-\frac{u^2}{2}c_s+\int_{\mathbb{R}} 
(e^{iu x}-1-iu x\mathbf{1}_{|x|\leq 1})F_s(dx)\right ]$.
Indeed $b: [0,T]\rightarrow \R$,
 $c$: $[0,T]\rightarrow \R_+$ are deterministic functions and
 for any $t\in [0,T]$, $F_t$ is a positive measure
such that\\ $\int_{[0,T] \times \R} (1\wedge x^2) F_t(dx) da_t<\infty$. 
For more details we refer to the statement and the proof of Proposition II.2.9 of~\cite{JS03}. 
\item The Borel measure on
${[0,T] \times \R} $ defined by $F_t(dx) da_t$ is called {\it jump measure} 
and it is denoted by $\nu(dt,dx)$.
\item We have   $\int_{[0,T] \times \R} x^2 \nu(dt,dx) =
\E\left( \sum_{t \in [0,T]} (\Delta X_t)^2 \right) $ where 
$\Delta X_t = X_t - X_{t-}$ is the jump at time $t$ of the process $X$.
\item Suppose that $X$ is square integrable. 
Since previous sum of jumps is bounded by the square bracket at time $T$, i.e.
 $[X,X]_T$,
which is  integrable, it follows that  
$\E \left( \sum_{t \in [0,T]} (\Delta X_t)^2 \right) < \infty$. 
\end{enumerate}
\end{remarque}
\begin{remarque}\label{rem:moment:b}
\begin{enumerate}
\item The process X is square integrable if and only if for every
 $t \in [0,T]$, $u\mapsto \varphi_t(u)$ is of class $C^2$.
\item By \eqref{BBB}, $X$ is square integrable if and only if 
  $u\mapsto \Psi_t(u)$ is of class $C^2$, 
$t \in [0,T], u \in \shu$.
\item If $X$ is square integrable, the chain rule derivation implies 
\begin{eqnarray}
\mathbb{E}[X_t]&=&-i\Psi_t^{'}(0)\ , \quad
\mathbb{E}[X_t-X_s]=-i(\Psi_t^{'}(0)-\Psi_s^{'}(0)),\label{EspXts} \\
Var(X_t)&=&-\Psi^{''}_t(0)\ , \quad
\label{VarX}\\
Var(X_t-X_s)&=&-[\Psi^{''}_t(0)-\Psi^{''}_s(0)]  \ .\label{VarXts}
\end{eqnarray}
\end{enumerate}
\end{remarque} 
%
\begin{remarque} \label{RemarkRCher4}
Suppose that $X$ is a square integrable PII process.
We observe that it is possible to permute integral
 and derivative in the expression
 \eqref{eq:LevyKhinAbs}.
In fact consider $t\in [0,T]$.
We need to show that
\begin{equation} \label{DSSS}
 \frac{d}{du} 
 \int_{[0,t] \times \R} \nu(ds,dx) g(s,x;u) =
\int_{[0,t] \times \R} \nu(ds,dx) 
\frac{\partial g}{\partial u} (s,x;u),
\end{equation}
where $g(s,x;u) =  ix(e^{iu x}-\mathbf{1}_{|x|\leq 1})$.
We observe that 
\begin{eqnarray*}
\vert g(s,x;u) \vert^2 &=& \vert e^{iux}-\mathbf{1}_{\vert x\vert \leq 1}\vert^2=\vert \cos ux -\mathbf{1}_{\vert x\vert \leq 1}\vert^2+\vert \sin ux\vert^2=\mathbf{1}_{\vert x\vert>1}+4\mathbf{1}_{\vert x\vert \leq 1} (\sin \frac{ux}{2})^2 \\
& \leq & 4 (\frac{u^2}{2}\vee 1)(x^2\wedge 1)\ .
\end{eqnarray*}
Hence, we obtain that for any real interval $[a,b]$ and  for any $u\in [a,b]$, 
 \begin{eqnarray*}
\left \vert \frac{\partial g(s,x;u)}{\partial u} \right \vert &=&
\left \vert \frac{\partial }{\partial u}(e^{iu x}-1-iux\mathbf{1}_{|x|\leq 1})\right \vert 
=
 \left \vert ix (e^{iux}-\mathbf{1}_{\vert x\vert \leq 1})\right \vert 
\leq 
2 x^2 (\vert u\vert \vee 1 ) \\
& \le &   2b x^2 =: \gamma(s,x) \ .
\end{eqnarray*}
Consequently by finite increments theorem, Remark \ref{RemarkRCher}
5)  it follows that
$\int_{[0,T] \times \R} \nu(ds,dx) \gamma(s,x) < \infty.$
By the definition of derivative and Lebesgue dominated convergence 
theorem the result \eqref{DSSS} follows.
So
\begin{equation} \label{3.1bis} 
\Psi_t'(u)=i\int_0^t b_sda_s-u\int_0^t c_sda_s+\int_0^t 
\left(\int_{\mathbb{R}} ix(e^{iu x}-\mathbf{1}_{|x|\leq 1})F_s(dx)
\right)da_s\ ,\quad \textrm{for all}\ u\in\mathbb R \ .
\end{equation}
Moreover, since
$$
\left \vert \frac{\partial}{\partial u}\left(ix(e^{iux}-\mathbf{1}_{\vert x\vert \leq 1})\right)\right \vert=\left\vert x^2e^{iux}\right\vert\leq x^2
$$
we obtain similarly
\begin{equation} \label{EXi}
\Psi''_t(u)=-\int_0^tc_sda_s-\int_{[0,T]\times \R}x^2e^{iux}F_s(dx)da_s=
-\int_0^t \xi_s(u) da_s \ ,
\end{equation}
where $\xi_s(u) = c_s+\int_\R x^2e^{iux}F_s(dx)$.
In particular, for every $u\in \R$, $t\mapsto \Psi'_t(u)$ and $t\mapsto \Psi''_t(u)$ are absolutely continuous with respect to $da_s$.
\end{remarque}
\begin{remarque} \label{RQEFS}
Suppose that $X$ is square integrable. A consequence 
of Remark \ref{RemarkRCher4}
 is the following.
\begin{enumerate}
\item $t\mapsto \Psi^{'}_t(u)$ is continuous  for every $u \in \R$
and therefore bounded on $[0,T]$.
%
\item $t\mapsto \Psi^{''}_t(0)$ is continuous.
\end{enumerate}
\end{remarque}

\subsection{Structure condition for PII } \label{SSCPII}
Let $X=(X_t)_{t \in [0,T]}$ be a real-valued semimartingale with
independent  increments and $X_0 = 0$.
 From now on, $X$ will be supposed to be \textbf{square integrable}.

\begin{propo}\label{prop:SCPII}

\begin{enumerate}
\item $X$ is a special semimartingale with decomposition
$X = M + A$ with the following properties:
$\left\langle M\right\rangle_t=-\Psi^{''}_t(0)$ and $ A_t=-i\Psi^{'}_t(0)$. In particular $t \mapsto  -\Psi^{''}_t(0)$ is increasing and therefore
of bounded variation.
\item $X$ satisfies condition~(SC) of Definition~\ref{defSC} if and only if 
\begin{equation}
\label{SCPII}
 d\Psi^{'}_t(0)\ll d\Psi^{''}_t(0) \quad \textrm{and}\quad
{\displaystyle \int_0^T\left|\frac{d\Psi^{'}_s}{d\Psi^{''}_s}(0)\right|^2|d\Psi_s^{''}(0)|<\infty}\ .
\end{equation}
In that case
\begin{equation}\label{AMPII}
A_t=\int_0^t \alpha_sd\left\langle M\right\rangle_s \quad  \textrm{with}\quad \alpha_t=i\frac{d\Psi^{'}_t(0)}{d\Psi^{''}_t(0)}\quad \textrm{for all}\ t\in [0,T].
\end{equation}
\item Under condition~(\ref{SCPII}), FS decomposition exists (and it is unique)
 for every square integrable random variable. 
\end{enumerate}
\end{propo}
Before going into the proof of the above proposition, let us derive one implication on the validity of the (SC) in the L\'evy case.  
Let $X=(X_t)_{t \in [0,T]}$ be a real-valued L\'evy process, with $X_0=0$. 
We assume that $\mathbb{E}[|X_T|^2]<\infty$.
%
\begin{enumerate}
\item Since $X=(X_t)_{t \in [0,T]}$ is a L\'evy process then $\Psi_t(u)=t\Psi_1(u)$. In the sequel, we will use the shortened notation  $\Psi:=\Psi_1$. 
\item $\Psi$ is a function of class $C^2$ and $\Psi^{''}(0)=Var(X_1)$ which
is strictly positive if $X_1$ is non deterministic.
\end{enumerate}
Then by application of Proposition~\ref{prop:SCPII}, we get the following result.
\begin{corro}\label{cor:SCLevy}
Let $X=M+A$ be the canonical decomposition of the L\'evy process $X$.
 Then for all $t\in[0,T]$, 
\begin{eqnarray}\left\langle M\right\rangle_t=-t\Psi^{''}(0) \quad \textrm{and}\quad A_t=-it\Psi^{'}(0)\ .\label{MALevy}\end{eqnarray}
If $\Psi^{''}(0)\neq 0$ then  $X$ satisfies condition~(SC) of Definition~\ref{defSC} with 
\begin{equation}
A_t=\int_0^t \alpha d\left\langle M\right\rangle_s \quad  \textrm{with}\quad \alpha=i \frac{\Psi^{'}(0)}{ \Psi^{''}(0)}\quad \textrm{for all}\ t\in [0,T]\ .\label{AMPIIS}
\end{equation}
Hence, FS decomposition exists for every square integrable random variable. If $\Psi^{''}(0)=0$ then $(X_t)$ verifies condition (SC) if and only if $X_t \equiv 0$.
\end{corro}

\begin{proof}[{\bf Proof of Proposition~\ref{prop:SCPII}}]
\begin{enumerate}
\item
Let us first determine $A$ and $M$ in terms of the log-characteristic
 function of $X$. Using~(\ref{EspXts}) of Remark~\ref{rem:moment:b}, we get
$
\mathbb{E}[X_t|\mathcal{F}_s]  =  \mathbb{E}[X_t-X_s+X_s\ |\
 \mathcal{F}_s] = -i\Psi^{'}_t(0)+i\Psi^{'}_s(0)+X_s$, 
 then 
$\mathbb{E}[X_t+i\Psi^{'}_t(0)|\mathcal{F}_s] =X_s +i\Psi^{'}_s(0)$. 
 Hence, $(X_t+i\Psi^{'}_t(0))$ is a martingale and the canonical decomposition of $X$ follows 
$
X_t=\underbrace{X_t+i\Psi^{'}_t(0)}_{M_t}\underbrace{-i\Psi^{'}_t(0)}_{A_t} \ ,
$
where $M$ is a local martingale and $A$ is a locally bounded variation process thanks to 
the  semimartingale property of $X$.  
%
Let us now determine  $\langle M\rangle $, in terms of the
log-characteristic function of $X$. 
Using~(\ref{EspXts}) and~(\ref{VarXts}) of Remark~\ref{rem:moment:b}, yields 
\begin{eqnarray*}
%
 \mathbb{E}[M^2_t|\mathcal{F}_s] & = &
 \mathbb{E}[(X_t+i\Psi^{'}_t(0))^2|
\mathcal{F}_s] 
    =  \mathbb{E}[(M_s+X_t-X_s+i(\Psi^{'}_t(0)-\Psi^{'}_s(0)))^2|
\mathcal{F}_s]\ , \\
 & = & M_s^2+Var(X_t-X_s)
= M_s^2-\Psi^{''}_t(0)+\Psi^{''}_s(0) \ .
\end{eqnarray*}
Hence, $(M^2_t+\Psi^{''}_t(0))$ is a $(\mathcal{F}_t)$-martingale,
and point 1. is established. 
\item is a consequence of point 1. and of Definition~\ref{defSC}.
On the other hand
$A_t=\int_0^t \alpha_sd\left\langle M\right\rangle_s \quad 
 \textrm{with}\quad \alpha_t=i\frac{d \Psi^{'}_t(0)}{d_t
   \Psi^{''}_t(0)}
\quad \textrm{for}\ t\in [0,T]\ .
$
\item follows from  Theorem \ref{ThmExistenceFS}. In fact 
$K_T =  
 \displaystyle \int_0^T \left(\frac{d\Psi^{'}_s}{d \Psi^{''}_s}(0)
\right)^2d(-\Psi_s^{''}(0))$
is deterministic and in particular K is uniformly bounded.
\end{enumerate}
\end{proof}
Condition (SC) implies a significant necessary condition.
\begin{propo}\label{PSCPIII}
If $X$ satisfies condition (SC), then one of the two following properties hold.
\begin{enumerate}
\item $X$ has no deterministic increments.
\item If $X_b-X_a$ is deterministic then $X_u=X_a$, $\forall u \in [a,b]$.
\end{enumerate}
\end{propo}
\begin{proof} \
We suppose that (SC) is fulfilled and let $0 \leq a < b \leq T$ 
for which $X_b-X_a$ is deterministic. Consequently 
$-(\Psi^{''}_b(0)-\Psi^{''}_a(0))=Var(X_b-X_a)=0$. 
This implies that $X_t-X_a$ is deterministic for every $t \in [a,b]$. By (\ref{SCPII}), it follows that $\Psi^{'}_t(0)=\Psi^{'}_a(0)$, $\forall t \in [a,b]$. Hence, for any $t \in [a,b]$, we have $X_t-X_a=\E[X_t-X_a]=0$.
\end{proof}

The following technical result will be useful in the sequel. 
\begin{propo}\label{LSCP1}
If $X$ satisfies condition (SC), there is $\tilde{a}:[0,T]\rightarrow \R$ increasing such that $d\tilde{a}_t$ is equivalent to $-d(\Psi_t^{''}(0))$ and \eqref{eq:LevyKhinAbs} holds with $da_t$ replaced by $d\tilde{a}_t$. 
\end{propo}
\begin{proof}
Appendix.
\end{proof}
From now on, $a_t$ will be replaced by $-\Psi_t''(0)$. Equalities and inequalities will generally hold $d(-\Psi''_t(0))$ a.e. with respect to $t$. 

\begin{corro}
\label{LSCP}
We suppose that $X$ is square integrable and it fulfills (SC). 
Then for every $u\in \R$, $t\mapsto \Psi_t(u)$, $t\mapsto \Psi'_t(u)$ and $t\mapsto \Psi''_t(u)$, are a.c. w.r.t. $-\Psi''_t(0)$. 
\begin{enumerate}
\item In particular,
$$
\Psi_t^{'}(u)=\int_0^t\zeta_s(u)d(-\Psi_s^{''}(0)) \quad \textrm{and} \quad \Psi_t^{''}(u)=\int_0^t\xi_s(u)d(-\Psi_s^{''}(0))\ ,
$$
where
$$
\zeta_s(u)=ib_s-uc_s+\int_\R ix(e^{iux}-1_{\{\vert x\vert \leq 1\}}) F_s(dx)
 \quad \textrm{and} \quad \xi_s(u)=c_s+\int_\R x^2e^{iux}F_s(dx)\ .
$$
\item Setting $u=0$, we obtain $\xi_s(0)=c_s+\int_\R x^2F_s(dx)=1$, $d(-\Psi_s^{''}(0))$ a.e.
\end{enumerate}
\end{corro}
\begin{proof} \
It follows from Proposition \ref{LSCP1}, item 3. 
 of Remark \ref{RemarkRCher} and Remark \ref{RemarkRCher4}.
\end{proof}

\subsection{Examples}
\subsubsection{A Gaussian continuous process example} \label{SS331}
Let $\psi:[0,T]\rightarrow     \mathbb{R}$ be a continuous increasing
function, $\gamma:[0,T]\rightarrow \mathbb{R}$ be a bounded variation
function. 
 We set
$X_t=W_{\psi(t)}+\gamma(t)$, 
where $W$ is the standard Brownian motion on $\mathbb{R}$. Clearly, $X_t=M_t+\gamma(t)$, where $M_t=W_{\psi(t)}$, defines a continuous martingale, such that $\left\langle M\right\rangle_t=\left[M\right]_t=\psi(t)$. Since 
$X_t\sim \mathcal{N}(\gamma(t),\psi(t))$, for all $u \in \mathbb{R}$ and $t \in [0,T]$, we have $\Psi_t(u)=i\gamma(t)u-\frac{u^2\psi(t)}{2}$ 
which yields $\Psi^{'}_t(0)=i\gamma(t)$ and $\Psi^{''}_t(0)=-\psi(t)$. Taking into account Proposition \ref{prop:SCPII} 2, (SC) is verified if and only if
 $ \gamma \ll \psi$ and ${\displaystyle \frac{d\gamma}{d\psi} \in \mathcal{L}^2(d\psi)}$. This is of course always verified if $\gamma \equiv 0$. 
We have $A_t=\int_0^t\alpha_sd\left\langle M\right\rangle_s 
\quad \textrm{and}\quad \alpha_t=\left . \frac{d\gamma}{d\psi}\right\vert_t
\quad \textrm{for all}\ t\in [0,T]$.
\subsubsection{Processes with independent and stationary increments
 (L\'evy processes)  }
\label{sec:example:Levy}

We recall some  log-characteristic functions 
of  typical L\'evy processes. In this case we have 
 $\Psi_t(u) = t \Psi(u), t \in [0,T], x \in \R$.
\begin{enumerate}
\item \underline{Poisson Case:}
If $X$ is a Poisson process with intensity $\lambda$,
then for all $u\in \R$, $\Psi(u)=\lambda(e^{iu}-1)$, $\Psi^{'}(0) =  i\lambda$ and $\Psi^{''}(0) =  -\lambda$,  which yields $\alpha \equiv 1$.
\item \underline{NIG Case:}
This process was introduced by Barndorff-Nielsen in~\cite{bnh}.
If $X$ is a Normal Inverse Gaussian L\'evy process with 
$X_1 \sim NIG(\theta,\beta,\delta,\mu)$,  with $\theta>|\beta|>0$, $\delta>0$
then for all $u\in \R$,
$
\Psi(u)=\mu iu + \delta (\gamma_0-\gamma_{iu})\ ,
\quad\textrm{where}\quad \gamma_{iu}=\sqrt{\theta^2-(\beta+iu)^2}.
$
By derivation, one gets 
$\Psi^{'}(0) =  i\mu + \delta \frac{i\beta}{\gamma_0}$ and $\Psi^{''}(0) =  -\delta(\frac{1}{\gamma_0}+ \frac{\beta^2}{\gamma^3_0} )$
which yields
${\displaystyle \alpha \equiv i\frac{\Psi^{'}(0)}{\Psi^{''}(0)}=
\frac{\gamma^2_0(\gamma_0\mu+\delta\beta)}{\delta(\gamma^2_0+\beta)}\ .}
$
\item \underline{Variance Gamma case:} 
If $X$ is a Variance Gamma process with 
 $X_1 \sim VG(\theta,\beta,\delta,\mu)$ where  $\theta, \beta > 0, \delta \neq 0$, then for all $u\in \R$,
 The expression of the log-characteristic function  can be found in~\cite{Ka06}
 or also~\cite{livreTankovCont},  table IV.4.5 in the particular case
  $\mu = 0$. We have 
$ \Psi(u) = \mu iu + \delta Log\left(\frac{\theta}{\theta-\beta iu +
 \frac{u^2}{2}}\right)$, ${\rm Log}(z)=ln|z|+i{\rm Arg}(z)$, the ${\rm Arg}(z)$ being chosen in $]-\Pi,\Pi]$, 
being the complexe logarithm.
 After derivation it follows
$\Psi^{'}(0) =  i(\mu - \delta \beta)$ and $\Psi^{''}(0) =  \frac{\delta}{\theta} (\theta^2 - \beta^2)$, which yields  $\alpha \equiv
 {\displaystyle \frac{\mu - \delta \beta}{\theta^2 - \beta^2}
 \frac{\theta}{\delta}}\ .$
\end{enumerate}

\subsubsection{Wiener integrals of L\'evy processes}\label{SWIL}
We take $X_t=\int_0^t\gamma_sd\Lambda_s$, where $\Lambda$ is a square integrable
L\'evy process as in Section~\ref{sec:example:Levy} with $\Lambda_0=0$.  Then,
$\int_0^T\gamma_s d\Lambda_s$ is well-defined  at least when $\gamma \in
\shl^\infty([0,T])$. It is then possible to calculate the characteristic
function and the cumulative function of $\int_0^\cdot \gamma_sd\Lambda_s$.
Let $(t,z) \mapsto t \Psi_\Lambda (z),$ \        
 denoting the log-characteristic function 
 of $\Lambda$. 
\begin{lemme}\label{lemme27}
Let $\gamma:[0,T] \rightarrow \R  $ 
be a Borel bounded function.
The log-characteristic function of $X_t$
is such that for all $u\in\mathbb{R}$, 
$
\Psi_{X_t}(u)=\int_0^t\Psi_\Lambda(u\gamma_s)ds\ ,\quad \textrm{where}\quad 
\mathbb{E}[\exp(iu X_t)]=\exp\big(\Psi_{X_t}(u)\big).
$
%
In particular, for every $t\in [0,T]$, $u \mapsto \Psi_{X_t}(u)$ is of class $C^2$ and so $X_t$ is square integrable for any $t\in [0,T]$.
\end{lemme}
\begin{proof}
Suppose first that $\gamma$ is
  continuous, then $\int_0^T\gamma_sd\Lambda_s$ is the 
limit in probability of $\sum_{j=0}^{p-1}\gamma_{t_j}(\Lambda_{t_{j+1}}-
\Lambda_{t_j})$ where $0=t_0<t_1<...<t_p=T$ is a subdivision of $[0,T]$
 whose mesh converges to zero. Using the independence of the increments,
 we have 
\begin{eqnarray*}
\mathbb{E}\left[\exp\{i\sum_{j=0}^{p-1}\gamma_{t_j}
(\Lambda_{t_{j+1}}-\Lambda_{t_j})\}\right]
&=&\prod_{j=0}^{p-1}\mathbb{E}\left[\exp\{i\gamma_{t_j}
(\Lambda_{t_{j+1}}-\Lambda_{t_j})\}\right]
=\prod_{j=0}^{p-1}\exp\{\Psi_\Lambda(\gamma_{t_j})(t_{j+1}-{t_j})\}\ ,\\ \\
&=&\exp\{\sum_{j=0}^{p-1}(t_{j+1}-{t_j})\Psi_\Lambda(\gamma_{t_j})\}\ .
\end{eqnarray*}
This converges to $\exp\left(\int_0^T\Psi_\Lambda(\gamma_s)ds\right)$, when the
mesh
 of the subdivision goes to zero.\\
Suppose now that $\gamma$ is only bounded and
 consider,  using convolution, a sequence $\gamma_n$
of continuous functions, such that $\gamma_n \rightarrow \gamma$ a.e. 
and $\sup_{t \in [0,T]}|\gamma_n(t)|\leq \sup_{t \in
  [0,T]}|\gamma(t)|$. 
We have proved that
\begin{eqnarray}
\mathbb{E}\left[\exp\left(i\int_0^T\gamma_n(s)d\Lambda_s\right)\right]
=\exp\left(\int_0^T\Psi_\Lambda(\gamma_n(s))ds\right)\ .
\label{F11}
\end{eqnarray}
Now, $\Psi_\Lambda$ is continuous therefore bounded, so Lebesgue dominated 
convergence and continuity of stochastic integral imply the statement.
\end{proof}
\begin{remarque}\label{remarque27}
\begin{enumerate}
\item A similar statement was written with respect to the log cumulant generating function, see \cite{BS03}.
\item The proof works also when
 $\Lambda$ has no moment condition and $\gamma$ is a continuous function with bounded variation. Stochastic integrals are then defined using integration by parts.
\end{enumerate}
\end{remarque}
%
Since $\Psi_{\Lambda}$ is of class $C^2$ we have,
$
\Psi_t^{'}(u)=\int_0^t\Psi_\Lambda^{'}(u\gamma_s)\gamma_sds,
\quad \textrm{and}\quad
 \Psi_t^{''}(u)=\int_0^t\Psi_\Lambda^{''}
(u\gamma_s)\gamma^2_sds\ $.
So
\begin{equation} \label{EE27}
\Psi_t^{'}(0)=\Psi_\Lambda^{'}(0)
\int_0^t\gamma_sds, 
\quad \Psi_t^{''}(0)=\Psi_\Lambda^{''}(0)\int_0^t\gamma^2_sds
\quad \textrm{and}\quad 
\alpha_t = i \frac{\Psi_\Lambda^{'}(0)}{\Psi_\Lambda^{''}(0)} 
\frac{1_{\{\gamma_t \neq 0\}}}{\gamma_t}.
\end{equation}
\begin{remarque}\label{RPASC}
\begin{enumerate}
\item $Var(X_T)=-\Psi_\Lambda^{''}(0)\int_0^T\gamma_s^2ds$.
\item If $\Psi_\Lambda^{''}(0)=0$ then $Var(X_T)=0$ and so $Var(X_t)=0$,
 $\forall t \in [0,T]$ and so $X$ is deterministic. Consequently Condition (SC) is only verified if $X$ vanishes identically because of Proposition \ref{PSCPIII}.
\end{enumerate}
\end{remarque}
\begin{propo}\label{PASC}
Condition (SC) is always verified if $\Psi_\Lambda^{''}(0)\neq 0$.
\end{propo}
\begin{proof}
We take into account item 2. of Proposition \ref{prop:SCPII}. Let $0< s<t\leq T$, \eqref{EE27} implies
\begin{eqnarray*}
\Psi_t^{'}(0)-\Psi_s^{'}(0)&=&\Psi_\Lambda^{'}(0)\int_s^t\gamma_rdr=\int_s^t\frac{\Psi_\Lambda^{'}(0)}{\gamma_r}1_{\{\gamma_r \neq 0\}}\gamma^2_rdr=\int_s^t\left(-i\alpha_r\right)d\left(-\Psi_r^{''}(0)\right)
\end{eqnarray*}
where
$
\alpha_r=\frac{\Psi_\Lambda^{'}(0)}{\Psi_\Lambda^{''}(0)}\frac{i}{\gamma_r}1_{\{\gamma \neq 0\}}
$.
This shows the first point of \eqref{SCPII}. In particular
$
\left|\frac{d\Psi_t^{'}(0)}{d\Psi_t^{''}(0)}\right|=\frac{|\Psi_\Lambda^{'}(0)|}{-\Psi_\Lambda^{''}(0)}\frac{1}{\gamma_t}
$.
The second point of \eqref{SCPII} follows because
$
\int_0^T|i\alpha_r|^2d\left(\Psi_r^{''}(0)\right)=T\frac{|\Psi_\Lambda^{'}(0)|^2}{\left(-\Psi_\Lambda^{''}(0)\right)^2}<\infty
$.
\end{proof}
\section{Explicit F\"ollmer-Schweizer decomposition in the PII case}

\setcounter{equation}{0}

Let $X=(X_t)_{t \in [0,T]}$ be a semimartingale (measurable process)
 with independent increments with log-characteristic function 
$(t,u)\mapsto \Psi_t(u)$. We assume that $(X_t)_{t \in [0,T]}$ is
 square integrable.
In this section, we  first  evaluate an explicit Kunita-Watanabe decomposition of a random variable $H$ w.r.t. the martingale part $M$ of $X$. Later, we 
obtain the decomposition with respect to $X$. Before doing so, it is useful
 to introduce in the following preliminary subsection an \textit{expectation} operator and a \textit{derivative} operator related to $X$. 

From now on we will suppose the validity of the (SC) condition.

\label{sec:FSPII}
\subsection{On some expectation and derivative operators}\label{s61}



We first introduce the \textit{expectation} operator related to $X$. 
For $0 \le t \le T$,
let $\epsilon^X_{t,T}$ denote the complex valued function defined for all $u\in \R$ by 
\begin{equation}
\label{eq:et}
\epsilon^X_{t,T}(u):=\exp(\Psi_{T}(u)-\Psi_t(u))\ .
\end{equation}
In the sequel, to simplify notations, we will write $ \epsilon_{t,T}$ instead of $\epsilon^X_{t,T}$. \\
%
We observe that the function
$
(u,t) \mapsto \epsilon_{t,T}(u)
$ 
and 
$
(u,t)\mapsto \epsilon^2_{t,T}(u)
$ 
are uniformly bounded because the characteristic function is bounded.
%
The  lemma below shows that the function $ \epsilon_{t,T}$ is closely related to the conditional expectation. 
\begin{lemme}
\label{lem:expectation}
Let $H=f(X_T)$
 where $f$ is given as a Fourier transform, 
$
f(x):=\hat{\mu}(x):=\int_{\R} e^{iux} \mu(du)\ ,
$ 
of a (finite) complex measure $\mu$  defined on $\R$.  \\
Then, for all $t\in [0,T]$, 
$\E[f(X_T) \vert \mathcal{F}_t] = e_{t,T}(X_t)$
where for all $x \in \R$,
$$ e_{t,T}(x):=\widehat{ \epsilon_{t,T}\mu}(x)=\int_{\R} e^{iux}
 \epsilon_{t,T}(u) \mu(du)\ .
$$
\end{lemme}
\begin{proof}
First, we easily check that $\int_{\R}  \epsilon_{t,T}(u)\mu(du)<\infty$, since $\mu$ is supposed to be a finite measure. \\
Now, let us consider the conditional expectation $\E[f(X_T)\,\vert \mathcal{F}_t]
$. By Fubini's theorem, 
$$
\E[f(X_T)\,\vert \mathcal{F}_t]=\E\left[\int_{\R} e^{iuX_T} \,\mu(du)\,\vert\, \mathcal{F}_t\right]=\int_{\R} \mu(du)\E[e^{iuX_T} \,\vert\, \mathcal{F}_t]\ .
$$
Finally,  remark that by the independent increments property of $X$, we obtain 
$$
\E[e^{iuX_T}\,\vert \mathcal{F}_t]=
\E\left[ e^{iu(X_T-X_t)}e^{iuX_t} \vert\, \mathcal{F}_t\right]
=\exp\big (\Psi_T(u)-\Psi_t(u)\big ) e^{iuX_t}\ ,\quad \textrm{for all}\ u\in\R\ .
$$
\end{proof}

Now let us introduce the \textit{derivative} operator related to the PII $X$. 
Let $\delta^X_t$ denote the complex valued function defined for all $u\in \R$
 by the Radon-Nykodim derivative
\begin{equation}
\label{eq:dt}
\delta^X_t(u):=i\frac{d(\Psi'_t(u)-\Psi'_t(0))}{d\Psi^{''}_t(0)}\ ,
\end{equation}
which is well-defined by Corollary~\ref{LSCP}. 
In the sequel, to simplify notations, we will write $ \delta_t$ instead of $\delta^X_t$. \\
%
By \eqref{3.1bis} in Remark \ref{RemarkRCher} 4., we obtain
\begin{equation} 
\label{eq:dtbis}
 \delta_t(u)=iu c_t+\int_{\R}x(e^{iux}-1)
F_t(dx) \ .  
\end{equation}
The lemma below shows that the function $ \delta_{t}$ is closely related to the Malliavin derivative in the sense of~\cite{nunno}. 
\begin{lemme}
\label{lem:derivative}
Let $\eta$ be a finite complex measure defined on $\R$ with a finite first order moment and $g$ its Fourier transform, i.e. the complex-valued function such that for all $x \in \R$, 
$g(x)=\hat{\eta}(x):=\int_{\R} e^{iux} \eta(du)\ .$
\begin{enumerate}
\item $g$ is differentiable with  bounded derivative;
\item $ \delta_t(u)\eta(du)$ is a finite complex measure.
\item  For all $x\in \R$, 
\begin{equation}
\label{eq:lem:derivative}
\widehat{ \delta_t\eta}(x):=\int_{\R}e^{iux} \delta_t(u)\eta(du) =c_t {g}'(x)+\int_{\R}\big ({g}(x+y)-{g}(x)\big )yF_t(dy)\ .
\end{equation}
\end{enumerate}
\end{lemme}
\begin{proof}
Item 1. is obvious.
We prove item 2. i.e. that $\int_{\R}\vert  \delta_t\eta \vert (du)<\infty$. For this, notice that the following upper bound
 holds for all $u,x\, \in \R$, 
\begin{equation}
\label{eq:major}
\vert x(e^{iux}-1)\vert =2\vert x\vert\, \left\vert \sin\frac{ux}{2}\right\vert \leq 2(\vert u\vert\vee 1)( x^2\wedge \vert x\vert)\ .
\end{equation}
Now, using the expression~(\ref{eq:dtbis}) of $ \delta_t$ 
 yields 
\begin{equation}
\label{eq:major:dt}
\vert  \delta_t(u)\vert
\leq 
\sqrt{2}\,\left [c_t\vert u\vert +2(1+\vert u\vert) \int_{\R} x^2\,F_t(dx)\right]
\leq 
2\sqrt{2} (1+\vert u\vert) \ ,
\end{equation}
because by point 2. of Corollary~\ref{LSCP}, $\ c_t \le 1$ and
$\int_{\R}x^2F_t(dx) \le 1 $ $\ d(-\Psi''_t(0))-$a.e. 
Finally, \eqref{eq:major:dt} and
the fact that $\eta$ is supposed to have a finite first order moment
imply the result.\\
We go on with the proof of  point 3.
Now we can consider the Fourier transform $\widehat{ \delta_t\eta}$. 
Using Fubini's theorem and \eqref{eq:major:dt},
 we obtain the following expression 
\begin{eqnarray*}
\widehat{ \delta_t\eta}(x)&=&\int_{\R}  \delta_t(u)\eta(du)e^{iux} \\
&=& c_t\int_{\R} iu \,\eta(du)e^{iux}+\int_{\R} \left ( \int_{\R}
 \eta(du)e^{iu(x+y)}-\int_{\R}
 \eta(du)e^{iux}  \right) y F_t(dy)\ .
\end{eqnarray*}
\end{proof}

We are now in the position to state an explicit expression for the Kunita-Watanabe decomposition of some random variables of the form $H=f(X_T)$. 
To be more specific, we consider a random variable which is given 
as a Fourier transform of $X_T$, 
\begin{eqnarray}
\label{ES1}
H=f(X_T) \quad \textrm{with}\quad f(x)=\hat{\mu}(x)=\int_{\mathbb{R}}e^{iu x}\mu(du)\ ,\quad \textrm{for all}\ x\in\mathbb{R}
\end{eqnarray}
for some finite complex signed measure $\mu$. 
%

\subsection{Explicit elementary Kunita-Watanabe decomposition}
By Proposition~\ref{prop:SCPII}, $X$ admits the following semimartingale 
decomposition, $X_t=A_t+M_t$, where
\begin{eqnarray}  A_t=-i\Psi^{'}_t(0)\ 
\quad\textrm{and}\quad \left\langle M\right\rangle_t=-\Psi^{''}_t(0)\label{97bis}\ .\end{eqnarray}
\begin{propo}\label{propo521}
Let $H=f(X_T)$ where $f$ is of the form~(\ref{ES1}).
Then, $H$ admits the decomposition
$$
H=\E[H]+\int_0^TZ_t dM_t+O_T\ ,
$$
where $O$ is a square integrable $(\mathcal{F}_t)-$martingale such that $\langle O,M\rangle=0$  and
$$
V_t:=\E[H\,\vert\,\mathcal{F}_t]=e_{t,T}(X_t)\ , \quad\textrm{and}\quad  Z_t={d_{t,T}}(X_{t^-})\ ,
$$
where the complex valued functions $e_{t,T}$ and $d_{t,T}$ are defined for all $x\in \R$ by  
\begin{equation}
\label{eq:KW}
{e_{t,T}}(x):=\widehat{ \epsilon_{t,T}\mu}(x)=\int_{\R}  \epsilon_{t,T}(u)\mu(du)e^{iux} 
\quad\textrm{and}\quad 
 {d_{t,T}}(x):=\widehat{ \delta_t \epsilon_{t,T}\mu}(x)=\int_{\R}  \delta_t(u) \epsilon_{t,T}(u)\mu(du)e^{iux} \, ,
\end{equation}
with $\epsilon_{t,T}$ being  defined in~\eqref{eq:et} and $\delta_t$ being defined in~\eqref{eq:dt}. 
Moreover, $\E[\int_0^TZ_s^2 d\langle M\rangle_s]<\infty \ .$
\end{propo}
In particular, 
$V_0=\mathbb{E}[H]\,.$
\begin{remarque}\label{rque:biblio1}
We remark that the (SC) condition is not a restriction
when $X$ is a martingale, since it is obviously fulfilled.
This would correspond to the classical Kunita-Watanabe statement.
\end{remarque}
\begin{remarque}\label{rque:biblio2}
In~\cite{nunno}, they obtain a similar decomposition valid for a different class of random variables. On  one hand
 their class is more general, allowing for path dependent payoffs,
 on the other hand it requires some stronger regularity assumptions
 since $H$ is supposed to be in the Malliavin-Sobolev space  $\D^{1,2}$. In our case, their regularity assumption on the payoff function
 could be relaxed by applying the \textit{derivative operator} $ \delta_t$ after applying the \textit{expectation operator} $ \epsilon_{t,T}$ whereas in~\cite{nunno}, they take the conditional expectation of the payoff Malliavin derivative.  \\
This trick of switching the conditional expectation and the differentiation is also implicitly used in the approach developed in~\cite{JS00} or similarly in~\cite{ContTankov}.  Their approach relies on the application of Ito's lemma on the conditional expectation $\E[H\vert\mathcal{F}_t]$ and therefore  requires some regularity conditions. Basically  the conditional expectation should be once
 differentiable w.r.t. the time variable and twice
 differentiable w.r.t. the space variable with continuous partial derivatives. On the other hand, their method is valid for a large class of martingale processes $X$.  
\\
Besides, our approach only relies on the martingale property of $(e^{iuX_t-\Psi_t(u)})_{0\leq t\leq T}$. Hence, $X$ is not required to be martingale as in~\cite{nunno},~\cite{JS00} or~\cite{ContTankov} and no specific regularity assumption on the payoff function or on the conditional expectation are required. Our approach is unfortunately restricted to additive processes. However, this specific setting allows to go one step further in providing an explicit expression for both the Follmer-Schweizer decomposition and the variance optimal strategy, as we will see below. Moreover, the expression of the Kunita-Watanabe decomposition derived in this specific case is quasi-explicit involving a simple Fourier transform. 

If $ \epsilon_{t,T}\mu$ admits a first order moment, then taking
$\eta = \epsilon_{t,T}\mu$, in Lemma~\ref{lem:derivative}, the conditional expectation function $e_{t,T}$  is differentiable w.r.t. the variable $x$
and we obtain 
\begin{equation}
\label{eq:KW:h:reg}
 {d_{t,T}}(x):=\widehat{ \delta_t \epsilon_{t,T}\mu}(x)=\int_{\R}  \delta_t(u) \epsilon_{t,T}(u)\mu(du)e^{iux}
 =c_t e'_{t,T}(x)+\int_{\R}\big (e_{t,T}(x+y)-e_{t,T}(x)\big )yF_t(dy)\ .
\end{equation}
%
The following lemma gives a condition on characteristics $c_t$ and $F_t$ ensuring the
 differentiability of $e_{t,T}$. 
%
\begin{lemme}
\label{lem:expectationReg}
Let $X$ be a PII process with finite second order moments such that there exist positive reals $\beta \in (0,2)$ and $\alpha$ verifying 
\begin{equation}
\label{eq:KallCond}
\inf_{t\in[0,T)} \left (
c_t+\int_{\vert x\vert \leq \vert u\vert ^{-1}} x^2\,F_t(dx)\right )
\geq \alpha\, \vert u\vert^{-2+\beta} \ , \quad\textrm{when} \ \vert u\vert \,\rightarrow\,\infty \ .
\end{equation}
Let $\mu$ be a finite complex measure defined on $\R$ and $f$ its Fourier transform  such that for all $x \in \R$, 
$f(x)=\hat{\mu}(x)\ .$ 
Then, for all $t\in [0,T)$, $ \epsilon_{t,T}\mu$ is a finite complex measure with finite moments of all orders and all the  derivatives of all orders 
of
  $x\mapsto e_{t,T}(x):=\widehat{ \epsilon_{t,T}\mu}(x)$
are  well-defined and bounded. 
\end{lemme}
\begin{remarque}\label{RKall}
When  $X$ is a L\'evy process,
 Assumption~(\ref{eq:KallCond}) implies the Kallenberg condition
stated in \cite{Kall81} ensuring the existence of a transition density
 for a L\'evy process $X$. 
\end{remarque}
\begin{proof}[Proof of Lemma~\ref{lem:expectationReg}]
We prove that $\int_{\R}u^p  \epsilon_{t,T}(u)\mu(du)<\infty$, for any 
nonnegative integer $p$. For this, we 
recall that Remark  \ref{RemarkRCher}   together with the lines below Proposition
\ref{LSCP1} say that
 for all $u\in\R$ and $t \in [0,T]$
\begin{equation} \label{425bis}
\vert  \epsilon_{t,T}(u)\vert 
= 
\exp \left\{-\frac{u^2}{2}\int_t^T c_s\,d(-\Psi''_s(0))\right\}\,
\left \vert \, \exp \left \{\int_t^T \int_{\mathbb{R}}
(e^{iu x}-1-iu x\mathbf{1}_{|x|\leq 1})
\, F_s(dx) \,d(-\Psi''_s(0))\right \}\right \vert\ .
\end{equation}
Consider now the second exponential term on the right-hand side of the above equality; it gives 
\begin{eqnarray}
\label{eq:epsMajor}
\left \vert 
\exp \big \{\int_t^T \int_{\mathbb{R}} 
(e^{iu x}-1-iu x\mathbf{1}_{|x|\leq 1})
\, F_s(dx) \,d(-\Psi''_s(0))
\big \} \right \vert
&\le &
\exp \big \{-2\int_t^T \int_{\mathbb{R}} 
(\sin\frac{ux}{2})^2
\, F_s(dx) \,d(-\Psi''_s(0))
\big \} \\
&\leq&
\exp \big \{-2\int_t^T \int_{\vert x\vert \leq \frac{\pi}{\vert u\vert}} 
(\sin\frac{ux}{2})^2
\,F_s(dx) \,d(-\Psi''_s(0))
\big \}\nonumber \\
&\leq& \exp \big \{-2\int_t^T \left (\frac{u}{\pi}\right)^2\,\int_{\vert x\vert \leq \frac{\pi}{\vert u\vert}} 
x^2
\, F_s(dx) \,d(-\Psi''_s(0))
\big \}\nonumber \ .
\end{eqnarray}
Hence, we conclude that for all $u\in\R$, 
\begin{equation}
\label{eq:major:etT}
\vert  \epsilon_{t,T}(u)\vert
\leq 
\exp \big \{-\frac{u^2}{2}\int_t^T \left [c_s+\frac{4}{\pi^2}\,\int_{\vert x\vert \leq \frac{\pi}{\vert u\vert}} 
x^2
\, F_s(dx) \right ]\,d(-\Psi''_s(0))
\big \}\ .
\end{equation}
Then by Assumption~(\ref{eq:KallCond}), there exists two positive reals $\alpha$ and $\beta$ such that
\begin{eqnarray*}
\vert  \epsilon_{t,T}(u)\vert
&\leq& 
 \exp\{-\alpha\,\left (\frac{u}{\pi}\right)^{\beta}\,(\Psi''_t(0)-\Psi''_T(0))\}\ , \quad\textrm{as} \ \vert u\vert \,\rightarrow\,\infty \ .
\end{eqnarray*}
Finally, we can conclude that under Assumption~(\ref{eq:KallCond}), for any nonnegative integer $p$, the complex measure $u^p \epsilon_{t,T}\mu$ is  finite with a bounded Fourier transform $g^{(p)}_{t,T}$, the $p$ order derivative of $e_{t,T}$. 
\end{proof}

\end{remarque}

\begin{remarque}\label{KWOU}
Notice that the explicit expression of the Kunita-Watanabe decomposition obtained in the case of an additive process  can be used to derive an explicit expression in the case where $X$ is an  Ornstein-Uhlenbeck process.
 Indeed if we consider
\begin{equation}
\label{eq:OU}
X_t=e^{-\alpha t} \tilde{X}_t\ ,
\end{equation}
for a given positive real 
$\alpha \in \R$ and additive process $\tilde X$. 
Consider a function $f$ satisfying condition~\eqref{ES1}.
We define now $\tilde f: [0,T] \times \R \rightarrow \R$
by $ \tilde f(t, \tilde x) = f(e^{-\alpha t} \tilde x),$ 
for all $\tilde x \in \R$, so that
$
f(X_T)=\tilde{f}(T,\tilde{X}_T)$.
Then by application of Proposition~\ref{propo521}, we get 
$
f(X_T)=\E[f(X_T)]+\int_0^T\tilde{Z}^T_sd\tilde{M}_s+O_T\ ,
$ 
where $\tilde{M}$ is the martingale part of $\tilde{X}$. 
Now
$
dX_t
=-\alpha e^{-\alpha t} \tilde{X}_tdt +e^{-\alpha t} d\tilde{M}_t-ie^{-\alpha t}d\Psi'_t(0)\ ,
$
where $\Psi$ is the log-characteristic function of $\tilde X$.
By uniqueness of the  Doob-Meyer decomposition of the special semimartingale
$X$,  the martingale part of $X$ is $M_t=\int_0^t e^{-\alpha s} d\tilde{M}_s$
 and 
finally we deduce the Kunita-Watanabe decomposition 
\begin{equation}
\label{eq:KWX}
f(X_T)=\E[f(X_T)]+\int_0^T Z_s dM_s+O_T\ ,\quad \textrm{with}\ Z_t=e^{\alpha t} \tilde{Z}_t^T\ .
\end{equation}
This can be easily generalized when $\alpha t$ is replaced by $\alpha(t)$ a
bounded  deterministic function of $t$. 
\end{remarque}

\begin{proof}
[\textbf{Proof of Proposition~\ref{propo521}}]
%

Lemma~\ref{lem:expectation} says that $V_t:=\E\left[H\vert \mathcal{F}_t\right]=\int V_t(u)d\mu(u)$ where 
\begin{equation}
\label{V}
V_t(u)=\epsilon_{t,T}(u)e^{iuX_t}=\E\left[\exp\left(iuX_T\right)\vert \mathcal{F}_t\right]\ .
\end{equation}
Having observed that $\vert\epsilon_{t,T}(u)\vert \leq 1$, for all $u\in \R$, we get
\begin{equation}\label{C10}
\sup_{t\leq T,u\in \R}\E\left[\vert V_t(u)\vert ^2\right]\leq 1\ .
\end{equation}
This implies that $V$ is an $(\mathcal{F}_t)$-square integrable martingale since $\mu$ is finite. We define
\begin{equation}\label{C11}
Z_t=\int_\R Z_t(u)d\mu(u)\ ,
\end{equation}
where
\begin{equation}
\label{ZZS}
Z_t(u)=\delta_{t,T}(u)\epsilon_{t,T}(u)e^{iuX_{t-}}\ .
\end{equation}
In the second part of this proof, we will show that
\begin{equation}\label{C12}
\E\left[\int_0^T\vert Z_t(u)\vert ^2d(-\Psi^{''}_t(0))\right]\leq 2\ .
\end{equation}
This implies in particular that process $Z$ in \eqref{C11} is well defined and $\E\left[\int_0^T\vert Z_s\vert ^2d\langle M\rangle_s\right]<\infty\ .$ 
We define 
\begin{equation}
\label{eq:O}
O_t:=V_t-V_0-\int_0^tZ_sdM_s\ .
\end{equation}
By additivity $O$ is an $(\mathcal{F}_t)$-square integrable martingale. It remains to prove that $\langle O,M\rangle=0$. For this, we will show that
$$
\langle V,M\rangle =\int_0^t Z_sd\langle M\rangle_s \ ,
$$
which will follow from the fact that $V_tM_t-\int_0^tZ_sd\langle M\rangle_s $ is an $(\mathcal{F}_t)$-martingale. 
In order to establish the latter, we prove that for every $0<r<t$,
\begin{equation}\label{C13}
\E\left[\left(V_tM_t-V_rM_r-\int_r^tZ_sd\langle M\rangle_s\right)R_r\right]=0
\end{equation}
for every bounded  $(\mathcal{F}_r)$-measurable variable
 $R_r$. Taking into account \eqref{C10} and \eqref{C12}, by Fubini's theorem,
 the left hand side of \eqref{C13} equals
\begin{equation}\label{C15}
\int d\mu(u)\E\left[\left(V_t(u)M_t-V_r(u)M_r-\int_r^tZ_s(u)d(-\Psi^{''}_s(0))\right)R_r\right]\ .
\end{equation}
It remains now to show that the expectation in \eqref{C15} vanishes for $d\mu(u)$ almost all $u$. Below, we will show that 
$$
V_t(u)M_t-\int_0^tZ_s(u)d(-\Psi^{''}_s(0))
$$
is an $(\mathcal{F}_t)$-martingale, for $d\mu(u)$ almost all $u$.
This implies that \eqref{C15} is zero.


We evaluate $\mathbb{E}[{V_t}M_t|\mathcal{F}_s]$.
 Since ${V}$ and $M$ are $(\mathcal{F}_t)$-martingales,
 using the property of independent increments we get
\begin{eqnarray*}
\mathbb{E}[{V_t(u)}M_t|\mathcal{F}_s]
&=&\mathbb{E}[{V_t(u)}M_s|\mathcal{F}_s]+\mathbb{E}[ {V_t(u)}(M_t-M_s)|
\mathcal{F}_s]\ ,\\\\
&=&M_s {V_s(u)}+ {V_s(u)}\mathbb{E}[\exp\{iu (X_t-X_s)-(\Psi_t(u )-\Psi_s(u ))
\}(M_t-M_s)]\ ,\\\\
&=&M_s {V_s(u)}+ {V_s(u)}e^{-(\Psi_t(u )-\Psi_s(u ))}
\mathbb{E}[e^{iu (X_t-X_s)}(M_t-M_s)]\ .
\end{eqnarray*}
Consider now the expectation on the right hand side of the above equality:
\begin{eqnarray*}
\mathbb{E}[e^{iu (X_t-X_s)}(M_t-M_s)]&=&\mathbb{E}[e^{iu (X_t-X_s)}(X_t-X_s)]+\mathbb{E}[e^{iu (X_t-X_s)}i(\Psi_t^{'}(0)-\Psi_s^{'}(0))]\ ,\\\\
&=&-i\frac{\partial}{\partial u }\mathbb{E}[e^{iu (X_t-X_s)}]+i(\Psi_t^{'}(0)-\Psi_s^{'}(0))\mathbb{E}[e^{iu (X_t-X_s)}]\ ,\\\\
&=&-ie^{\Psi_t(u )-\Psi_s(u )}(\Psi^{'}_t(u )-\Psi^{'}_s(u ))
+i(\Psi_t^{'}(0)-\Psi_s^{'}(0))e^{\Psi_t(u )-\Psi_s(u )}\ .\end{eqnarray*}
Consequently,
\begin{eqnarray*}
\mathbb{E}[ {V_t(u)}M_t|\mathcal{F}_s]
&=&
M_s {V_s(u)}-i {V_s(u)}(\Psi^{'}_t(u )-\Psi^{'}_s(u ))+i {V_s(u)}
(\Psi_t^{'}(0)-\Psi_s^{'}(0))\\ \\
&=&M_s {V_s(u)}-i {V_s(u)}\left(\Psi^{'}_t(u )-\Psi_t^{'}(0)-
(\Psi^{'}_s(u )-\Psi_s^{'}(0))\right)\ .
\end{eqnarray*}
This implies that $\left(  {V_t(u)}M_t+i {V_t(u)}(\Psi^{'}_t(u )-
\Psi_t^{'}(0))\right )_t$ is an $(\mathcal{F}_t)$-martingale.
Then by integration by parts, 
$$
 {V_t(u)}(\Psi^{'}_t(u )-\Psi_t^{'}(0))=\int_0^t {V_s(u)}\,
d(\Psi^{'}_{s}(u )-\Psi_{s}^{'}(0))+\int_0^t(\Psi^{'}_{s}(u )-
\Psi_{s}^{'}(0))d {V_s}(u)\ .
$$
The proof is concluded once we have shown \eqref{C12}.

By Cauchy-Schwarz inequality 
%
\begin{eqnarray*}
\E \left( \int_0^T Z^2_s \,d(-\Psi''_s(0)) \right)
&=&\E \left(\int_0^T \left \vert \int_{\R}  \delta_s(u)
 \epsilon_{s,T}(u) e^{iuX_s}\mu(du)  \right \vert^2\,d(-\Psi''_s(0))\right) \\
&\leq&
\vert \mu\vert (\R)\, \int_0^T \int_{\R} \vert  \delta_s(u) \epsilon_{s,T}(u)\vert ^2\,\vert  \mu\vert (du) \,d(-\Psi''_s(0))\\
&=&
\vert \mu\vert (\R)\,\int_{\R} \vert  \mu\vert (du) \int_0^T \vert  \delta_s(u) \epsilon_{s,T}(u)\vert ^2  \,d(-\Psi''_s(0))\ .
\end{eqnarray*}

Let us consider now, for a given real $u$ the integral
 w.r.t. to the time parameter in the right-hand side of the above inequality.
 Using inequalities (\ref{eq:epsMajor}) and \eqref{425bis}, we obtain 
$$
\vert \epsilon_{s,T}(u)\vert ^2 
\leq -\frac{1}{2}
\exp \left \{ \int_s^T -\Big [u^2 c_r
+4 \int_{\R} (\sin \frac{ux}{2})^2
\,F_r(dx)\Big ]\,d(-\Psi^{''}_r(0))
\right \}\ .
$$
On the other hand, by \eqref{eq:dtbis} and \eqref{eq:major} we have 
\begin{eqnarray*}
\vert \delta_s(u)\vert ^2
&\leq & 
2 \left [u^2c_s^2 + 4 \Big (\int_{\R} \vert x \sin \frac{ux}{2}\vert  \,F_s(dx) \Big )^2\right ]\\
&\leq &
2\left [u^2c_s^2 + 4 \int_{\R}  x^2\,F_s(dx)\,\int_{\R} (\sin \frac{ux}{2})^2  \,F_s(dx) \right ]\ .
\end{eqnarray*}
Since $c_s\leq 1$, $\int_\R x^2 F_s(dx)\leq 1$, by item 2.
 of Corollary \ref{LSCP}, we finally get
\begin{equation}
\label{eq:deltaMajor}
\vert \delta_s(u)\vert ^2 \leq  2 \gamma_s(u) 
\end{equation}
where 
\begin{equation} \label{434gamma}
\gamma_s(u)=u^2c_s + 4 \int_{\R} (\sin \frac{ux}{2})^2  \,F_s(dx).
\end{equation}
 Then
\begin{eqnarray} \label{B11}
\int_0^T \vert  \delta_s(u) \epsilon_{s,T}(u)\vert ^2  \,d(-\Psi^{''}_s(0))
&\leq&
2\int_0^T \gamma_s(u)
\exp \left \{ \int_s^T -\gamma_r(u)\,d(-\Psi^{''}_r(0))
\right \}
d(-\Psi^{''}_s(0))  \nonumber \\
&& \\
&=&
2\left (1-\exp(\big \{ \int_0^T -\gamma_r(u)\,d(-\Psi^{''}_r(0))
\big \}\right ) \le 2_ .\nonumber
\end{eqnarray}
\end{proof}

\begin{example}
We take $X=M=W$ the classical Wiener process with canonical filtration ($\shf_t$). We have $\Psi_s(u )=-\frac{u ^2s}{2}$ so that $\Psi^{'}_s(u )=-u s$ and $\Psi^{''}_s(u )=-s$. So $Z_s(u)=iu V_s(u)$. We recall that 
$V_s=\mathbb{E}[\exp(iu W_T)|\mathcal{F}_s]=\exp(iu W_s)\exp\left(-u ^2\frac{T-s}{2}\right)\ .$
In particular,  $V_0=\exp(-\frac{u ^2T}{2})$ and so
$\exp(iu W_T)=i\int_0^Tu \exp(iu W_s)
\exp\left(-u ^2\frac{T-s}{2}\right) dW_s + 
\exp(-\frac{u ^2T}{2}).$
In fact that expression is classical and 
it can be derived from Clark-Ocone formula. In fact, if $D$ is the usual Malliavin derivative then $\E\left(D_t\exp(iuW_T)|\shf_t\right)=iu\exp(iuW_s-\frac{u^2}{2}(T-s))$.
\end{example}

\subsection{Explicit F\"ollmer-Schweizer decomposition}

We are now able to evaluate the FS decomposition of
 $H=f(X_T)$ where $f$ is given by~(\ref{ES1}). First, we state the following lemma. 
 
%
\begin{lemme}
\label{lem:FS}
For all $s,t\in [0,T)$, 
\begin{equation}
\label{eq:lem:FS}
\int_s^t Re \big(i\delta_r(u)d\Psi'_{r}(0)\big )
\leq 
K_T+\int_s^t \int_{\mathbb{R}} \left (\sin \frac{ux}{2}\right)^2 F_r(dx)\,d(-\Psi''_r(0)),
\end{equation}
where the process $K$ was defined in Definition~\ref{D29}.
\end{lemme}

%
\begin{proof} Using \eqref{eq:dtbis} and \eqref{3.1bis}, 
with a slight abuse of notation, it follows
\begin{eqnarray*}
Re \big (i\delta_r(u)d\Psi'_{r}(0)\big )&=&
-\left ( b_r+\int_{\vert x\vert > 1} x F_r(dx)\right ) \,
\int_{\mathbb{R}}\left ( x(\cos (ux) -1)F_r(dx)\right ) \,d(-\Psi''_r(0))\\
&=&2\,\left ( b_r+\int_{\vert x\vert > 1} x F_r(dx)\right ) \, \int_{\mathbb{R}}\left ( x\big (\sin \frac{ux}{2} \big)^2 F_r(dx)\right ) \,d(-\Psi''_r(0))\\
&\leq &\left [ \Big ( b_r+\int_{\vert x\vert > 1} x F_r(dx)\Big )^2  +\Big ( \int_{\mathbb{R}} x\big (\sin \frac{ux}{2} \big)^2 F_r(dx)\Big )^2 \right ]\,d(-\Psi''_r(0))\ .
\end{eqnarray*}
Indeed, by the (SC) condition using Proposition~\ref{prop:SCPII}
and Corollary \ref{LSCP} 1., we obtain that for all $t\in [0,T)$, 
$$
\int_s^t\Big ( b_r+\int_{\vert x\vert > 1} x F_r(dx)\Big )^2 \,d(-\Psi''_r(0)) 
=
\int_s^t \left \vert \frac{d\Psi'_r(0)}{d(-\Psi''_r(0)} \right 
\vert ^2 d(-\Psi''_r(0))
=
K_t - K_s \le K_T,
$$
which is a deterministic bound.
Finally, recalling that $\int_{\mathbb{R}} x^2 F_r(dx)\leq 1$ 
by Corollary \ref{LSCP} 2.,
 Cauchy-Schwarz inequality implies 
$$
\Big ( \int_{\mathbb{R}} x\big (\sin \frac{ux}{2} \big)^2 F_r(dx)\Big )^2
\leq \int_{\mathbb{R}} \big (\sin \frac{ux}{2} \big)^2 F_r(dx)
\ .
$$

\end{proof}




%
\begin{thm}\label{thmCO}
The FS decomposition of $H=f(X_T)$ where $f$ satisfies~(\ref{ES1})
is the following
\begin{equation}
\label{E3240}
H_t=H_0+\int_0^t \xi_sdX_s+L_t \quad\textrm{with}\quad H_T=H\ , 
\end{equation}
where
\begin{equation}
\label{eq:FS:ltT1}
\xi_t={k_{t,T}}(X_{t^-})\ ,\quad \textrm{with}\quad {k_{t,T}}(x)=\int_{\R} e^{ i\int_t^T  \delta_s(u)\,d\Psi'_s(0)} \, \delta_t(u) \epsilon_{t,T}(u)    e^{iux}\mu(du)  \ ,
\end{equation}
and
\begin{equation}
\label{eq:FS:htT1}
H_t={h_{t,T}}(X_{t})\ ,\quad\textrm{with}\quad {h_{t,T}}(x)=\int_{\R} e^{ i\int_t^T  \delta_s(u)\,d\Psi'_s(0)} \, \epsilon_{t,T}(u) e^{iux} \mu(du)  \ ,
\end{equation}
with $\epsilon_{t,T}$  defined in~\eqref{eq:et} and $\delta_t$ defined in~\eqref{eq:dt}. 
%
\end{thm}
\begin{proof}
Let us introduce the following notations, which will correspond to the expression~\eqref{eq:FS:htT1} for $H_t$ and~\eqref{eq:FS:ltT1} for $\xi_t$ in the 
case where $\mu=\delta_u$ for a given real $u$:
\begin{equation}
\label{EA2}
H_t(u):=e^{ i\int_t^T  \delta_s(u)\,d\Psi'_s(0)} \, \epsilon_{t,T}(u)e^{iuX_t}\quad\textrm{and}\quad 
\xi_t(u):=e^{ i\int_t^T  \delta_s(u)\,d\Psi'_s(0)} \, \delta_t(u) \epsilon_{t,T}(u)e^{iuX_{t{^-}}}\ .
\end{equation}
\begin{enumerate}
\item 
%
We first introduce the process $H$.
Taking into account Lemma~\ref{lem:FS} together with
 inequalities (\ref{eq:epsMajor}) and \eqref{425bis},
 $|H(u)|$ is uniformly bounded in $u$ and $t$. Indeed
\begin{eqnarray}
\label{eq:H:Major}
\vert H_t(u)\vert &=&\vert e^{ i\int_t^T  \delta_r(u)\,d\Psi'_r(0)} \vert \, 
\vert \epsilon_{t,T}(u)\vert \nonumber \\
&\leq &
\exp\left \{K_T+\int_t^T \int_{\mathbb{R}} \left (\sin \frac{ux}{2}\right)^2 F_r(dx)\,d(-\Psi''_r(0))\right \} \nonumber \\
&&
\exp\left \{\int_t^T -\frac{1}{2}\Big [u^2 c_r+4\int_{\mathbb{R}} \left (\sin \frac{ux}{2}\right)^2 F_r(dx)\Big ]\,d(-\Psi''_r(0))\right \} \nonumber \\
& \leq & \exp(K_T)\, \exp\left \{-\frac{1}{4}\int_t^T \gamma_r(u) \,d(-\Psi''_r(0))
\right \} \\
&\leq &
\exp(K_T), \nonumber 
\end{eqnarray}
where $\gamma$ was defined in \eqref{434gamma}. \\
By Fubini's, 
\begin{equation}\label{DefH}
H_t = \int_\R H_t(u) d\mu(u)
\end{equation}
 is well-defined and it equals the expression in \eqref{eq:FS:htT1}. 
We  prove now that $\xi$ defined in~\eqref{eq:FS:ltT1} is a well-defined square integrable process.
Using the above bounds~(\ref{eq:H:Major}) and~(\ref{eq:deltaMajor}) we obtain
\begin{eqnarray}
\label{eq:xiMajor}
|\xi_t(u)|^2
&\leq&
\vert \delta_t(u)\vert ^2 \vert H_t(u)\vert^2  \nonumber \\
&\leq & 2\gamma_t(u)\,\exp(2K_T)\,\exp\left \{-\frac{1}{2} \int_t^T \gamma_r(u) \,d(-\Psi''_r(0))\right \},
\end{eqnarray}
which finally implies 
\begin{equation}
\label{DEA1}
\E\left(\int_0^T|\xi_t(u)|^2d\left\langle M\right\rangle_t\right )
\leq 4\left (1-\exp\left \{ \int_t^T \gamma_r(u) \,d(-\Psi''_r(0))\right \} \right)\leq 4\ .
\end{equation}
 Hence, $\xi(u)\in \Theta := L^2(M)$ for any $u\in \R$.
 (\ref{DEA1}), \eqref{B11} yield 
\begin{equation}
\label{DEA2}
\int_\R d\mu(u)\int_0^T\E(|\xi_t(u)|^2)d\left\langle M\right\rangle_t <\infty\ .
\end{equation}
The above upper bound implies that $\xi_t = \int_\R \xi_t(u) d\mu(u)$
is well-defined and $\xi\in \Theta $ 
and it equals the expression \eqref{eq:FS:ltT1}.
Consequently $\xi\in L^2(M)=\Theta$ and using  stochastic and classical Fubini's we get
\begin{equation}\label{EEXi}
\int_0^t \xi_s dX_s =\int_{\R} d\mu(u)\int_0^t \xi_s(u) dX_s\ .
\end{equation}
\item We go on with the proof of Theorem \ref{thmCO} showing the following:\\
a) $L_t=H_t-H_0-\int_0^t\xi_sdX_s$ is an eventually complex valued square integrable martingale;\\
b) $\left\langle L,M\right\rangle=0$ where $M$ is the martingale part of the special semimartingale $X$.
\item We first establish a) and b) for the case $\mu$ is the Dirac measure at some fixed 
 $u\in\R$. We will show that
\begin{eqnarray}
\label{EEDecH}
H_t(u)=H_0(u )+\int_0^t  \xi_s(u)dX_s+L_t(u )
\quad \textrm{with}\quad H_T(u )=\exp(iu X_T)\ ,\end{eqnarray}
for fixed $u \in\mathbb{R}$
where $L(u)$ is a square integrable martingale and 
$\langle L(u),M \rangle  = 0$.
Notice that by relation (\ref{EA2}), 
$H_t(u)= e^{i\int_t^T\delta_s(u)d\Psi'_s(0)} V_t(u)$ with $V_t(u )=e^{iu X_t}\epsilon_{t,T}(u)$, 
as introduced in~\eqref{V}. Integrating by parts, gives
%
\begin{eqnarray}
\label{EEDecV}
H_t(u)=H_0(u) + \int_0^t e^{i\int_r^T\delta_s(u)d\Psi'_s(0)} V_r(u)
\big (-i\delta_r(u)d\Psi'_r(0)\big )+\int_0^t e^{i\int_r^T \delta_s(u)d\Psi'_s(0)} dV_r(u)\ .\end{eqnarray}
We denote again by $Z(u)$ the expression provided by $(\ref{ZZS})$. 
We observe that 
$$
\xi_t(u)=e^{i\int_t^T \delta_s(u)d\Psi'_s(0)}Z_t(u).
$$
We recall that 
\begin{eqnarray}\label{EEDecO}
dV_r(u)=Z_r(u)dM_r+dO_r(u)=Z_r(u)(dX_r-dA_r)+dO_r(u)\ ,
\end{eqnarray}
where $A$ is given by~$(\ref{97bis})$ and $O(u)$ is a square
 integrable martingale  strongly 
orthogonal to $M$. 
Replacing (\ref{EEDecO}) in (\ref{EEDecV}) yields
\begin{eqnarray*}
H_t(u )&=&H_0(u )+L_t(u )+\int_0^t e^{i\int_r^T\delta_s(u)d\Psi'_s(0)} Z_r(u )dX_r\\
&+& i\int_0^t e^{i\int_r^T\delta_s(u)d\Psi'_s(0)}Z_r(u )d\Psi'_{r}(0)
-i\int_0^t e^{i\int_r^T\delta_s(u)d\Psi'_s(0)}V_r(u )\delta_r(u)d\Psi'_r(0) \\
&=& H_0(u )+L_t(u) + \int_0^t \xi_s(u) dX_s,
\end{eqnarray*}
where
\begin{equation}
\label{EA12}
L_t(u)=\int_0^t e^{i\int_r^T\delta_s(u)d\Psi'_s(0)} dO_r(u).
\end{equation}
$L(u)$ is a local martingale which is also a square integrable martingale
because
$ \int_0^T e^{2Re(i \int_t^T \delta_s(u) d\Psi'_s(0)} d\langle O \rangle_t   $            is finite 
taking into account Lemma \ref{lem:FS}.

Since $O(u)$ is strongly orthogonal with respect to $M$,
then $L(u)$ has the same property.



%
\item We treat now the general case discussing the points a) and b) in item 2.  $(\ref{EEXi})$ and the Definition of $H$ 
show that
\begin{equation}\label{EEL1}
L_t:=H_t-H_0-\int_0^t\xi_sdX_s
\end{equation}
fulfills
\begin{equation}\label{EEL2}
\int_{\R}  L_t(u)d\mu(u) =L_t,
\end{equation}
for every $t \in [0,T]$.
Let $0\leq s < t \leq T$ and $R_s$ a bounded $\shf_s$-measurable random variable.
Using~\eqref{eq:H:Major},~\eqref{DEA1} and Cauchy-Schwarz we obtain
\begin{eqnarray} \label{EstB}
\E[(\vert L_t(u)\vert ^2)&=& \E\Big (\left \vert H_t(u)-H_0(u)-\int_0^t \xi_r(u)dX_r\right \vert ^2 \Big)\nonumber\\
&\leq&2 \E\left(\vert H_t(u)\vert ^2\right)+4\E\left(\vert H_0(u)\vert ^2\right)+8\left(\E\left(\int_0^t \xi_r(u)dM_r\right)^2+\E\left(\int_0^t \xi_r(u)\alpha_rd [M]_r\right)^2\right)\nonumber\\
&\leq&2 \E\left(\vert H_t(u)\vert ^2\right)+4\E\left(\vert H_0(u)\vert ^2\right)+8\left(1+K_T\right)\E\left(\int_0^t \vert \xi_r(u)\vert ^2d(-\Psi^{''}_r(0))\right)\nonumber\\
&\leq&6\exp(2K_T)+32(1+K_T).
\end{eqnarray}

\begin{enumerate}
\item By \eqref{EstB}, we observe that
$\E \left( \int_\R d\mu(u) \vert L_t(u)\vert \right) <\infty$.
 Fubini's,~(\ref{EEL2}) and the fact that $L(u)$ is an $(\shf_t)$-martingale give $\E[L_tR_s]=\E[L_sR_s]$. Therefore $L$ is an $(\shf_t)$-martingale. For every $t \in [0,T]$, $(L_t)$ is a square integrable because of $(\ref{EEL1})$ and by additivity.
\item 
By item 3. $L(u) M$ is an $(\shf_t)$-local martingale. Moreover $L(u)$ and $M$ are square integrable martingales.
By Cauchy-Schwarz and Doob inequalities, it follows that $E(\sup_{t\in [0,T]} \vert L_t(u) M_t \vert)$
is finite. Consequently $L(u)M$ is indeed an  $(\shf_t)$-martingale.
 It remains to show that $LM$ is an  $(\shf_t)$-martingale.
This is a consequence of Fubini's provided we can justify
\begin{equation}\label{PFub1}
\E(L_tM_tR_s)=\int_{\R}d\mu(u)\E[L_t(u)M_tR_s].
\end{equation}
For this we  need to estimate 
\begin{equation}\label{PMu1}
\int_{\R}d\mu(u)\E(|L_t(u)M_tR_s|)\ .
\end{equation}
By Cauchy-Schwarz the square of expression \eqref{PMu1} is bounded by
$$
\vert \vert R \vert \vert_{\infty}\int_\R d\mu(u)\E\left(\vert L_t(u)\vert^2\right)\int_\R d\mu(u)\E\left(\vert M_t\vert^2\right)\leq \vert \mu \vert(\R)^2\vert \vert R \vert \vert_{\infty}\E\left(\vert M_T\vert^2\right)\sup_{t\leq T;u\in \R}\E\left(\vert L_t(u)\vert^2\right)
$$
\eqref{PMu1}  follows by \eqref{EstB}.
This finally shows that the expression $(\ref{E3240})$ in the statement of Theorem \ref{thmCO} is an FS type decomposition which could be theoretically complex.
\end{enumerate}
\item It remains to prove that the decomposition is real-valued.
 Let $(H_0,\xi,L)$ and $(\overline{H_0},\overline{\xi},\overline{L})$
 be two FS decomposition of $H$. Consequently,
since $H$ and $(S_t)$ are real-valued, we have
\begin{eqnarray*}0=H-\overline{H}=(H_0-\overline{H}_0)+\int_0^T(\xi_s-\overline{\xi}_s)dX_s+(L_T-\overline{L}_T)\ ,\end{eqnarray*}
which implies that $0=Im(H_0)+\int_0^TIm(\xi_s)dX_s+Im(L_T)$.
 By Theorem \ref{ThmExistenceFS}, the uniqueness of the real-valued
 F\"ollmer-Schweizer decomposition yields
 that the processes $(H_t)$,$(\xi_t)$ and $(L_t)$ are real-valued.
\end{enumerate}
\end{proof}

\section{The error in the quadratic minimization problem}\label{sec4}
\setcounter{equation}{0}


Let $H\in \mathcal{L}^2$. The problem of minimization of the quadratic error given in Definition \ref{Defproblem2} is strongly connected with the FS decomposition. 
 We evaluate now the error committed by the mean-variance hedging procedure.
%
%
In the following lemma, we first calculate $\left\langle L(u),L(v)\right\rangle$ for any $u,v\in \R$.
\begin{lemme}\label{PCrochetuv}
We have
\begin{equation}
\label{CrochetLuv}
\left\langle L(u),L(v)\right\rangle_t
=
\int_0^t\epsilon_{t,T}(u)\epsilon_{t,T}(v)
e^{i\int_t^T\big (\delta_r(u)+\delta_r(v)\big )d\Psi'_r(0)}d\Gamma_{s}(u,v)\ ,
\end{equation}
where $(V_t(u))$ is the exponential martingale defined by
$
V_t(u)=e^{iuX_t}\epsilon_{t,T}(u)\ , 
$
as introduced in~\eqref{V} and 
\begin{equation}
\label{eq:Gamma}
\Gamma_{t}(u,v)=\nu_{t}(u,v)-\int_0^t\delta_s(u)\delta_s(v)d(-\Psi^{''}_s(0))\ ,\quad \textrm{with}
\end{equation}
\begin{equation}\label{FEC2bis}
\nu_t(u,v)=\Psi_t(u+v)-\Psi_t(u)-\Psi_t(v)
\end{equation}
\end{lemme}
\begin{proof}
We have
\begin{eqnarray*}
L_t(u)&=&H_t(u)-H_0(u)-\int_0^t\xi_r(u)dX_r\\
&=&V_t(u)e^{i\int_t^T \delta_s(u)d\Psi'_s(0)}-e^{\Psi_T(u)+i\int_0^T \delta_s(u)d\Psi'_s(0)}-\int_0^t\xi_r(u)dM_r-\int_0^t\xi_r(u)dA_r\ .
\end{eqnarray*}
Using integration by parts and the fact that $t\mapsto \int_0^t \delta_s(u)d\Psi'_s(0)$ is continuous (since $t\mapsto \Psi^{''}_t(0)$ is),
\begin{equation}\label{ENM1}
L_t(u)=\shm_t(u)+\sha_t(u)\ ,
\end{equation}
where
\begin{equation}\label{EMM}
\shm_t(u)=\int_0^t e^{i\int_r^T\delta_s(u)d\Psi'_s(0)}dV_r(u)-\int_0^t\xi_r(u)dM_r-e^{\Psi_T(u)+i\int_0^T\delta_s(u)d\Psi'_s(0)}
\end{equation}
\begin{equation*}
\sha_t(u)=-i\int_0^t e^{i\int_r^T\delta_s(u)d\Psi'_s(0)}V_r(u)\delta_r(u)d\Psi'_r(0)+i\int_0^t\xi_r(u)d\Psi'_r(0)\ .
\end{equation*}
We observe that $\sha_t(u)$ is predictable, $\sha_0(u)=0$, $L_0(u)=0$. By uniqueness of the
 decomposition of an ($\shf_t$)-special semimartingale, we obtain 
\begin{equation}\label{ENM12}
L(u)_t=\shm_t(u)=\int_0^tZ_s(u)d(-\Psi^{''}_s(0))
\end{equation}
Since $\langle O(u),M\rangle=0$ where $O(u)$ was defined in~\eqref{eq:O}, it follows
\begin{equation}\label{EMZ}
\left\langle V(u),M\right\rangle_t=\int_0^tZ_s(u)d\left\langle M\right\rangle_s\ ,
\quad\textrm{where}\quad Z_t(u)=\delta_t(u)V_t(u)\ .
\end{equation}
We need at this point to express the predictable covariation $\left\langle V(u),V(v)\right\rangle$, $\forall u,v \in \R$. For this we decompose the product $V(u)V(v)$ to obtain
\begin{equation}\label{VVR}
V_t(u)V_t(v)=V_t(u+v)R_t(u,v)\ ,
\end{equation}
where
$$
R_t(u,v)=\frac{\epsilon_{t,T}(u)\epsilon_{t,T}(v)}{\epsilon_{t,T}(u+v)}=\exp\{-(\nu_{T}(u,v)-\nu_{t}(u,v))\}\ .
$$
Since $(R_t(u,v))_{t\in[0,T]}$ is continuous, integrating by parts we obtain
$$
V_t(u)V_t(v)=\int_0^tR_s(u,v)dV_s(u+v)+\int_0^tV_s(u+v)R_s(u,v)d\nu_{s}(u,v)\ .
$$
Since $(\int_0^tR_s(u,v)dV_s(u+v))_t$ is an $(\shf_t)$-local martingale, it follows that
\begin{equation}\label{EVCroc}
\left\langle V(u),V(v)\right\rangle_t=\int_0^tV_s(u+v)R_s(u,v)(d\Psi_{s}(u+v)-d\Psi_{s}(u)-d\Psi_{s}(v))\ .
\end{equation}
We come back to the calculus of $\left\langle L(u),L(v)\right\rangle_t$; (\ref{EMM}) and (\ref{ENM12}) give
\begin{eqnarray*}
\left\langle L(u),L(v)\right\rangle_t
&=&\left\langle L(u),\int_0^.e^{i\int_r^T\delta_s(v)d\Psi'_s(0)}dV_r(v)\right\rangle_t\\
&=&\left\langle \int_0^.e^{i\int_r^T\delta_s(u)d\Psi'_s(0)}dV_r(u),\int_0^.e^{i\int_r^T\delta_s(v)d\Psi'_s(0)}dV_r(v)\right\rangle_t\\
&&-\left\langle \int_0^.\xi_r(u)dM_r,\int_0^.e^{i\int_r^T\delta_s(v)d\Psi'_s(0)}dV_r(v)\right\rangle_t\\
&=&\int_0^te^{i\int_r^T\big (\delta_s(u)+\delta_s(v)\big )d\Psi'_s(0)}d\left\langle V(u),V(v)\right\rangle_r-\int_0^t\xi_r(u)e^{i\int_r^T\delta_s(v)d\Psi'_s(0)}d\left\langle M,V(v)\right\rangle_r
\end{eqnarray*}
Using $(\ref{EMZ})$, $~(\ref{VVR})$ and $(\ref{EVCroc})$, we obtain 
\begin{eqnarray*}
\left\langle L(u),L(v)\right\rangle_t
&=&\int_0^t e^{i\int_r^T\big (\delta_s(u)+\delta_s(v)\big )d\Psi'_s(0)}
V_r(u+v)R_r(u,v)d\nu_{r}(u,v)\\
&&-\int_0^t V_r(u)V_r(v)
e^{i\int_r^T\big (\delta_s(u)+\delta_s(v)\big )d\Psi'_s(0)}\delta_r(u)
\delta_r(v)d(-\Psi^{''}_r(0))\\
&=&\int_0^t e^{i\int_r^T\big (\delta_s(u)+\delta_s(v)\big )d\Psi'_s(0)}\epsilon_{r,T}(u)\epsilon_{r,T}(v) \left[d\nu_{r}(u,v)+\delta_r(v)\delta_r(u)d\Psi^{''}_r(0)\right]\ ,
\end{eqnarray*}
which concludes the proof.
\end{proof}
Now we can evaluate the error committed by the mean-variance hedging procedure described at Section~\ref{sec4}. 
\begin{thm}\label{TErrorGlobal}
Let $X=(X_t)_{t \in [0,T]}$ be a semimartingale with independent increments
with log-characteristic function $\Psi$.
 Then the variance of the hedging error equals $J_0:=\int_{\R^2}d\mu(u)d\mu(v)J_0(u,v)$
where
$$
J_0(u,v)=\int_0^T\exp\left(\int_t^T \Big(\frac{d\Psi^{'}_s(0)}{d\Psi^{''}_s(0)}\Big)^2d\Psi^{''}_s(0)
+i\int_t^T \big (\delta_s(u)+\delta_s(v) \big )d\Psi_s'(0)
\right)
\epsilon_{t,T}(u)\epsilon_{t,T}(v)d\Gamma_{t}(u,v)\ ,
$$
where $\epsilon_{t,T}$  is defined in~\eqref{eq:et}, $\delta_t$ is defined in~\eqref{eq:dt} and $\Gamma $ is defined 
 in~\eqref{eq:Gamma}.
\end{thm}
\begin{proof}
Theorem \ref{ThmSolutionPb1} implies that the variance of the hedging error equals
\begin{equation}\label{E421}
\E \left(\int_0^T\exp\{-(K_T-K_s)\}d\left\langle L\right\rangle_s \right)\ ,
\end{equation}
where $K$ was defined in Definition $\ref{D29}$. By (\ref{AMPII}), it follows that
$
K_t=\int_0^t\left(\frac{d\Psi^{'}_s(0)}{d\Psi^{''}_s(0)}\right)^2d(-\Psi^{''}_s(0))
$.
We come back to  expression~(\ref{CrochetLuv}) given in Lemma~\ref{PCrochetuv}. It gives
\begin{equation}\label{ELLuv0}
\left\langle L(u),L(v)\right\rangle_t=C^1(u,v,t)+C^2(u,v,t)
\end{equation}
where
\begin{equation}\label{EELC1}
C^1(u,v,t):=
\int_0^t e^{i\int_s^T \big (\delta_r(u)+\delta_r(v) \big )d\Psi_r'(0)}
\epsilon_{s,T}(u)\epsilon_{s,T}(v)
e^{i(u+v)X_s} 
d\nu_{s}(u,v)
\end{equation}
\begin{equation}\label{ELLC2}
C^2(u,v,t):=
\int_0^t e^{i\int_s^T \big (\delta_r(u)+\delta_r(v) \big )d\Psi_r'(0)}
\epsilon_{s,T}(u)\epsilon_{s,T}(v)\delta_s(u)\delta_s(v)
e^{i(u+v)X_s} 
d\Psi''_s(0)\ .
\end{equation}
We need to show that
\begin{equation}\label{EMuLuVar}
\E\left(\int_{\R^2}d|\mu|(u)d|\mu|(v)||\left\langle L(u),L(v)\right\rangle||_{var}\right)<\infty\ .
\end{equation}
%
%
Observe that 
$$
\Vert C^1(u,v,\cdot)\Vert_{var}=\int_0^T \vert H_s(u)H_s(v)\vert d\vert \nu_s\vert (u,v)\ .
$$
Let us consider first the term involving the measure $\nu$. Notice that 
$$
d\nu_s(u,v)=\left [-uvc_s+\int_{\R}(e^{iux}-1)(e^{ivx}-1)F_s(dx)\right ] d(-\Psi^{''}_s(0))\ ,
$$
Then, we obtain the following upper bound 
\begin{eqnarray*}
d\vert\nu_s\vert(u,v)&\leq &\left [\vert uv\vert c_s+2 \int_{\R}\vert \sin\frac{ux}{2}\vert\,\vert\sin\frac{vx}{2}\vert F_s(dx)\right ] d(-\Psi^{''}_s(0))\\
&\leq & \frac{1}{2}\left [(u^2+v^2)c_s+2\int_{\R} \Big [(\sin \frac{ux}{2})^2+(\sin \frac{vx}{2})^2 \Big ]F_s(dx)\right ] d(-\Psi^{''}_s(0))\\
&\leq &  \frac{1}{2} \big [\gamma_s(u)+\gamma_s(v)\big ]\ ,
\end{eqnarray*}
where $\gamma_s$ is defined in~(\ref{434gamma}). 
we obtain by setting $\gamma_s(u,v)=(u^2+v^2)\frac{c_s}{2}+\int_{\R} (\sin \frac{ux}{2})^2(\sin \frac{vx}{2})^2 F_s(dx)$
Now, using inequality~\eqref{eq:H:Major} yields 
\begin{eqnarray*}
\Vert C^1(u,v, \cdot)\Vert_{var}&=&\frac{1}{2}\exp(2K_T)\,\int_0^T \left [\gamma_s(u)+\gamma_s(v)\right ] \, 
\exp\left \{-\frac{1}{4} \int_s^T \Big [\gamma_r(u)+\gamma_r(v)\Big ] d(-\Psi_r''(0))\right \}\, d(-\Psi^{''}_s(0))\\
&=&2\exp(2K_T)\,\Big (1
-\exp\left\{-\frac{1}{4} \int_0^T \big [\gamma_r(u)+\gamma_r(v)\big ] d(-\Psi_r''(0))\right \}\Big )\\
&\leq& 2\exp(2K_T)\ ,
\end{eqnarray*}
which implies, by the fact that $\mu$ is finite 
\begin{equation}
\label{eq:C1}
\int_{\R^2} d|\mu|(u)d|\mu|(v)||C^1(u,v,.)||_{var}\leq 
2  \exp(2K_T) \vert \mu \vert(\R)^2.
\end{equation}
Now, using~\eqref{eq:xiMajor}, 
\begin{eqnarray}
\label{D11bis}
\Vert C^2(u,v,\cdot)\Vert_{var}
&\leq & 
\int_0^T   \vert \xi_s(u)\vert \,\vert  \xi_s(v) \vert d(-\Psi''_s(0)) \nonumber \\
&\leq &
2\exp(2K_T)\int_0^T \sqrt{\gamma_s(u)\gamma_s(v)}\exp\left \{-\frac{1}{4} \int_s^T \big (\gamma_r(u)+\gamma_r(v)\big )\,d(-\Psi''_r(0))\right \}  \,  d(-\Psi''_s(0))\nonumber \\
&\leq &
-4\exp(2K_T)\int_0^T \left (-\frac{1}{4}\right )\big (\gamma_s(u)+\gamma_s(v)\big )\exp\left \{-\frac{1}{4} \int_s^T \big (\gamma_r(u)+\gamma_r(v)\big )\,d(-\Psi''_r(0))\right \}  \,  d(-\Psi''_s(0)) \nonumber \\
&\leq&4\exp(2K_T)\ .
\end{eqnarray}
Finally \eqref{D11bis} implies
\begin{equation}\label{EMuLuCVar}
\int_{\R^2}\left\|C^2(u,v,.)\right\|_{var}d|\mu|(u)d|\mu|(v) \le 4 \exp(2 K_T) \vert \mu \vert(\R)^2.
\end{equation}
Since $K_T$ is deterministic,
 \eqref{EMuLuCVar} and \eqref{eq:C1} imply \eqref{EMuLuVar}.
Previous considerations allow to prove that
\begin{equation}\label{EELLL}
\left\langle L \right\rangle_t=\int_{\R^2}d\mu(u)d\mu(v)
\left\langle L(u),L(v) \right\rangle_t.
\end{equation}
For this, it is enough to show that
\begin{equation}\label{EELLL1}
L_t^2-\int_{\R^2}d\mu(u)d\mu(v) \left\langle L(u),L(v) \right\rangle_t
\end{equation}
produces an ($\shf_t$)-martingale. First (\ref{EMuLuVar}) shows
 that the second term in (\ref{EELLL1}) is well-defined.
 By \eqref{EstB} and Fubini's,
 (\ref{EELLL1}) gives
$
\int_{\R^2}d\mu(u)d\mu(v)(L_t(u)L_t(v)- \left\langle L(u),L(v) \right\rangle_t).
$
By similar arguments as point 4.(a) in the proof of Theorem \ref{thmCO}, using the fact that $(L_t(u))$ is a martingale and 
applying Fubini's, we are able to show that (\ref{EELLL1}) defines a martingale.
According to  (\ref{E421}), the last step of the proof consists in evaluating the expectation of
$
\int_0^Texp\{-(K_T-K_s)\}d\left\langle L\right\rangle_s
$
taking into account (\ref{EELLL}) and Lemma~\ref{PCrochetuv}.
\end{proof}

\section{Examples}

\label{ESGM}
\subsection{The Gaussian examples}\label{SGM} 
We refer here to the toy model introduced at Section \ref{SS331}. We suppose that $X_t=\gamma(t)+W_{\psi(t)}$ with $\psi$ increasing,
 $d\gamma \ll d\psi$ and $\frac{d\gamma}{d\psi} \in \mathcal{L}^2(d\psi)$. This guarantees the (SC) property because of Proposition \ref{prop:SCPII} 2.
Given $f$ and $\mu$ expressed via (\ref{ES1}), the FS-decomposition of $H=f(X_T)$ is provided by Theorem \ref{thmCO} with
$$
\Psi_t^{'}(0)=i\gamma(t)\ , \quad \delta_t(u)=iu\ ,\quad \textrm{and}\quad \epsilon_{t,T}(u)= \exp[iu(\gamma(T)-\gamma(t))-\frac{u^2}{2}(\psi(T)-\psi(t))]\ ,
$$
which yields 
$$
  H_t(u)=\exp[2iu(\gamma(T)-\gamma(t))-\frac{u^2}{2}(\psi(T)-\psi(t))]e^{iuX_t} \quad \textrm{and} \quad
\xi_t(u)=iuH_t(u)
$$
According to Lemma~\ref{PCrochetuv}, we can easily show that $\Gamma_t(u,v)\equiv 0$, for all $t \in [0,T], u,v \in \R$. Consequently, the variance of the hedging error is zero.

\subsection{The L\'evy case}\label{S5.2}
Let $X$ be a square integrable L\'evy process, with characteristic function
$\Psi_t(u)$ where $\Psi_t(u)=t\Psi(u)$ .
It is always a semimartingale since $\Psi\rightarrow e^{i\Psi_t(u)}$ has bounded variation, see Theorem 4.14 of \cite{JS03}. By Remark \ref{rem:moment:b}, $\Psi$ is of class $C^2(\mathbb{R})$.
We suppose that $\Psi^{''}(0)\neq 0$.
 We have
$$
\frac{d\Psi^{'}_t(0)}{d\Psi^{''}_t(0)}=
\frac{\Psi^{'}(0)}{\Psi^{''}(0)}.
$$
Condition (SC) is verified taking into account Proposition \ref{prop:SCPII}.
In conclusion, we can apply Theorem \ref{thmCO} taking into account 
\eqref{eq:et} and \eqref{eq:dt}, we obtain $V_t(u )= \exp((T-t)\Psi(u ))e^{iuX_t},$
$$
 H_t(u )=\exp \left((T-t)\big (\Psi(u)+\frac{\Psi'(u )-\Psi'(0)}{\Psi''(0)}\Psi'(0)\big )\right)e^{iuX_t}\quad \textrm{and} \quad \xi_t(u )=H_t(u )i\frac{\Psi^{'}(u )-\Psi^{'}(0)}{\Psi^{''}(0)}\ .
$$
%
The factor $\Gamma_t(u,v)$ appearing in Lemma~\ref{PCrochetuv} gives $\Gamma_t(u,v)=t\Gamma(u,v)$ and
$$
\Gamma(u,v)=\left(\Psi(u+v)-\Psi(u)-\Psi(v)\right)+\frac{(\Psi^{'}(v)-\Psi^{'}(0))(\Psi^{'}(u)-\Psi^{'}(0))}{(-\Psi^{''}(0))}\ .
$$
In particular, when $X$ is a Poisson process we have $\Gamma(u,v) \equiv 0$.
This shows, as expected, that $\left\langle L(u),L(v)\right\rangle=0$, $\forall u,v\in \R$.
\subsection{Wiener integral of L\'evy processes}\label{SILC}
With the same notations as in subsection \ref{SWIL}, we consider a square integrable L\'evy process $\Lambda=(\Lambda_t)_{t \in [0,T]}$ such that $\Lambda_0=0$ 
and  $Var(\Lambda_1) \neq 0$.
Let $\gamma:[0,T]\rightarrow \R$ be a bounded Borel function. We set $X_t=\int_0^t \gamma_sd\Lambda_s$, $t \in [0,T]$. 
 For $u \in \R$, $t \in [0,T]$, we have the following quantities.
According to the observations below Remark \ref{remarque27}
\begin{equation}\label{EZETAW}
\epsilon_{t,T}(u)=\exp\left(\int_t^T\Psi_\Lambda(u\gamma_s)ds\right)
\end{equation}
\begin{equation}\label{EEtaW}
\delta_{t,T}(u)=-i  \frac{\Psi^{'}_{\Lambda}(u\gamma_t)-\Psi^{'}_{\Lambda}(0)}{\Psi^{''}_{\Lambda}(0)}.
\end{equation}
\begin{remarque}\label{RS12C}
If $\Lambda=W$ then there is a Brownian motion $\tilde{W}$ such that $X_t=\tilde{W}_{\Psi(t)} $ with $\Psi(t)=\int_0^t\gamma_s^2ds$. This was the object of Section \ref{SGM}.
\end{remarque}
\subsection{Representation of some contingent claims by Fourier transforms}\label{RSCFT}
In general, it is not possible to find a Fourier representation, of 
the form~(\ref{ES1}), for a given payoff function which is not 
necessarily bounded or integrable. Hence, it can be more convenient 
to use the bilateral Laplace transform that allows an extended domain
 of definition including non integrable functions.
 We refer to~\cite{Cramer},~\cite{Raible98} and more recently~\cite{Eberlein} for such characterizations of payoff functions.
It should be certainly possible to extend our approach
replacing the Fourier transform with the bilateral Laplace transform.
%
However, to illustrate the present approach 
 restricted to payoff functions represented as classical Fourier transforms, we give here one simple example of such representation extracted from~\cite{Eberlein}.  
 The payoff of a {\it self quanto put option} with strike $K$ is 
$$
f(x)=e^x(K-e^x)_+\quad \textrm{and}\quad \hat{f}(u)=\int_{\mathbb{R}} e^{iux} f(x)\,dx=\frac{K^{2+iu}}{(1+iu)(2+iu)}\ .
$$
In this case $\mu$ admits a density which is proportional to $\hat f$
which is integrable.

\section*{Appendix}

\begin{proof}[{\bf Proof of Proposition \ref{LSCP1}}]


In the sequel, we will make use of Lemma 3.12 of \cite{GOR} in 
a fairly  extended generality.

\begin{lemme}\label{GOR}
Let ${\cal N}$ be a complete metric 
 space and $\mu$ and $\nu$
are two non-negative Radon   non-atomic measures.
We suppose the following:
\begin{enumerate}
\item $\mu \ll  \nu\ ;$
\item $\mu(I)\neq 0$ for every open ball $I$ of ${\cal N}$.
\end{enumerate} 
Then $h := {\displaystyle \frac{d\mu}{d\nu} \neq 0}$ $\nu$ a.e. In particular
 $\mu$ and $\nu$ are equivalent.
 \end{lemme}

\begin{enumerate}
\item If there are no deterministic increments then setting $d\mu(t)=-d\Psi_t^{''}(0)$ and $d\nu(t)=da_t$, it follows that $da_t$ is equivalent to $-d\Psi_t^{''}(0)$, because of Lemma \ref{GOR}. Consequently the result is established.
\item Suppose that not all the increments are non-deterministic.
We decompose $E:= [0,T] = E_R \cup E_R^C$ where
$$
E_R=\left\{t\in [0,T] \vert Var(X_{(t+\varepsilon) \wedge T} - X_{t}) > 0,
 \forall
  \epsilon  > 0 \ {\rm or} \
 Var(X_t  - X_{(t-\varepsilon)_+}) \neq 0 \ 
 \forall   \epsilon  > 0
 \right \},
$$
and its complementary
$$
E_R^C=\left\{t\in [0,T]\vert \exists \varepsilon > 0, 
 Var(X_{(t+\varepsilon) \wedge T} - X_{(t-\varepsilon)_+} = 0\right\}.
$$
Without restriction to generality we can suppose that $T \in E_R$.
Since $E_R^C$ is an open subset of $[0,T]$, it can be decomposed
into a union $\bigcup_{n \in \N} I_n$ of open (disjoint) intervals of $[0,T]$.
We denote $a_n = \inf I_n, b_n = \sup_n I_n$.
Clearly $a_n$ and $b_n$ belong to $ E_R$ and to its boundary,
since $E_R$ is closed.
We define on $E$ a semidistance $d$ such that $d(u,v)=Var(X_v-X_u)$.
The equivalence relation $\shr$ on $E$ defined setting
$x \shr y$ if and only if $d(x,y) = 0$, produces the 
following equivalence classes:
$$ \{t\}, t \in {\rm int} E_R, \quad I_n, n \in \N.$$
The quotient $E/d$ can be identified with family of {\it typical
 representatives} $E_d = {\rm int} E_R  \bigcup \{ b_n, \ n \in \N \}.$ \\
We denote by $\tilde{a}_t=\int_{[0,t]\cap E_R}da_s$.
The proof of Proposition \ref{LSCP1} will be concluded if
the two lemmas below hold.
\begin{lemme}\label{Lequiv}
\begin{enumerate}
\item $\int_{E_R^C}d\left(-\Psi_t^{''}(0)\right)=0$.
\item   $  \Psi_t(u)$  (resp. $  \Psi_t^{''}(u)$)
 is absolutely continuous with respect to $d\tilde{a}$, for every $u \in \R$.
\item $\int_{E_R^C}d\tilde a_t=0$.
\end{enumerate}
\end{lemme}
\begin{lemme}\label{Lequiv1} $d\tilde{a}$ is equivalent
to $d\left(-\Psi_t^{''}(0)\right)$.
\end{lemme}
\begin{proof}[{\bf Proof of Lemma \ref{Lequiv}}]
\begin{enumerate}
\item Since each $I_n$ is precompact, it can be recovered by a
countable sequence of subintervals of the type $]t_n-\varepsilon_n, 
t_n+\varepsilon_n[  $. So, by definition of $E_R^C$,
 we have $\int_{I_n} d(-\Psi_t'')(0) = 0$.
The item  follows then because  $E_R^C$ is the union of countable intervals.
\item It is enough to show that for every $B$ Borel subset of $E_R^C$ we have
 $$\int_B \xi_s(u)da_s  = \int_B \eta_s(u)da_s=0,$$
where $\eta_s(u)$ (resp. $\xi_s(u)$) was introduced in \eqref{eq:LevyKhinAbs} (resp. \eqref{EXi}). 
We only treat the $\Psi_t(u)$ case, the other 
one being similar.
By Proposition
 \ref{PSCPIII}, if $X_b-X_a$ is deterministic then $X$
and in particular $t \mapsto \Psi_t(u)$
 is constant on $[a,b]$. 
Consequently
\begin{equation*}
\left\vert \int_{E^R_C} \eta_s(u)  da_s\right\vert \leq \sum_{n \in \N}
\left\vert \int_{I_n}   \eta_s(u)   da_s\right \vert=\sum_{n \in I_n}
\left\vert  \int_{I_n}
d\Psi_t(0) \right \vert = 0.  \end{equation*}
  \item It is a consequence of the definition of $\tilde a$.
\end{enumerate}
\end{proof}
\begin{proof}[{\bf Proof of Lemma \ref{Lequiv1}}]
${\cal N} := E_d$ is a complete  metric space equipped with the 
distance, inherited from $E$, still denoted by $d$.
We define $d\mu$ (resp. $d\nu$)  the measure on 
the Borel $\sigma$-algebra of ${\cal N}$
obtained by 
 restriction  from $-d\Psi_t^{''}(0)$  (resp. $d\tilde{a}_t$).
This is possible since 
by items (a) and (c) of Lemma \ref{Lequiv}
 $$\int_{E_R^C} d(-\Psi_t^{''}(0))= 
\int_{E_R^C} d\tilde a_s = 0.$$
Condition 1 
of  Lemma  \ref{GOR} is  verified by item (b).
Concerning Condition 2. of the same lemma,  let  $ t_0 \in E_d 
\subset E_R$ 
and $B(t_0)$ an open ball centered at $t_0$.
Obviously 
$\mu(B(t_0)) > 0$.
By  Lemma \ref{GOR}, $\nu  \sim \mu$ on the 
Borel $\sigma$-algebra of $ E_d$ and
the result  follows.
\end{proof}
\end{enumerate}
This concludes the proof of Proposition  \ref{LSCP1}.
\end{proof}


\bigskip
{\bf ACKNOWLEDGEMENTS:} The research of the third named author was partially 
supported by the ANR Project MASTERIE 2010 BLAN-0121-01. 



\begin{thebibliography}{9}
\bibitem{bnh} Barndorff-Nielsen, O.E. and Halgreen, C. (1977).
 \textit{Infinite divisibility of the hyperbolic and generalized inverse Gaussian distributions} Zeitschrift f\"ur Wahrscheinlichkeitstheorie und verwandte Gebiete, vol 38, 309-312.
\bibitem{Benth-Kallsen} Benth, F. E., Kallsen, J. and Meyer-Brandis,
  T. (2007).
  \textit{A non-Gaussian Ornstein-Uhlenbeck process 
for electricity spot price modeling and derivatives pricing},
 Applied Mathematical Finance, 14(2), 153-169.
\bibitem{nunno} Benth, F. E., Di Nunno, G.,  
L{\o}kka, A., {\O}ksendal, B. and 
 Proske, F. (2003).
{\it Explicit representation of the minimal variance portfolio in 
markets driven by L\'evy processes}, 
Conference on Applications of Malliavin Calculus in Finance 
(Rocquencourt, 2001).
Math. Finance 13, no. 1, 55--72. 
\bibitem{BS03} Benth, F.E. and Saltyte-Benth, J. (2004).
 \textit{The normal inverse Gaussian distribution and spot price modelling in energy markets}, International journal of theoretical and applied finance, vol 7(2), 177-192.
\bibitem{buckdahn} Buckdahn, R. (1993).
 \textit{Backward stochastic differential
equations driven by a martingale}, Unpublished.
\bibitem{livreTankovCont}  Cont, R. and Tankov, P. (2003).
 \textit{Financial modelling with jump processes}, Chapman \& Hall / CRC Press.
\bibitem{ContTankov}  Cont, R., Tankov, P.
and  Voltchkova, E. (2007).
\textit{Hedging with options in models with jumps},
  Stochastic analysis and applications,  197--217, 
Abel Symp., 2, Springer, Berlin.
\bibitem{Cramer}  Cramer, H. (1939). \textit{On the
 representation of a
    function by certain Fourier integrals},
 Transactions of the American
  Mathematical
 Society 46, 191-201.
\bibitem{nunno02} Di Nunno, G. (2002).
\textit{Stochastic Integral Representations, Stochastic Derivatives and Minimal
Variance Hedging},
Stoch. Stoch. Rep. 73, 181-198. 
\bibitem{BookDiNunno} Di Nunno, G., {\O}ksendal, B. and Proske, F. (2009).
\textit{Malliavin Calculus for L\'evy Processes with Applications to Finance}. Berlin, Heidelberg: Springer-Verlag.
\bibitem{Eberlein} Eberlein, E.,  Glau, C. and
 Papapantoleon, A. (2009).
 \textit{Analysis of Fourier transform valuation
 formula and applications}, preprint.
\bibitem{FS91} F\"ollmer, H. and Schweizer, M. (1991).
{\it Hedging of contingent claims under incomplete information},
  Applied stochastic analysis (London, 1989),  389-414,
 Stochastics Monogr., 5, Gordon and Breach, New York.
\bibitem{GOR} Goutte, S., Oudjane, N. and Russo, F. (2011).
\textit{Variance Optimal Hedging for continuous time additive processes.}
Preprint HAL inria-00437984,
http://fr.arxiv.org/abs/0912.0372.
\bibitem{GORDiscr} Goutte, S., Oudjane, N. and Russo, F. (2011).
\textit{Variance Optimal Hedging for discrete time processes
with independent increments. Applications to Electricity Markets}. \\
To appear: Journal of Computational Finance. 
Preprint HAL : inria-00473032.
http://hal.archives-ouvertes.fr/inria-00473032/fr/.
\bibitem{Ka06} Hubalek, F., Kallsen, J. and Krawczyk, L. (2006).
 \textit{Variance-optimal hedging for processes with stationary independent increments}, The Annals of Applied Probability, Volume 16, Number 2, 853-885.
\bibitem{JS00} Jacod, J.,  M\'el\'eard, S. and and  Protter, P. (2000). 
\textit{Explicit form and robustness of martingale representations}, The Annals of Probability, Volume 28, Number 4, 1747-1780. 
\bibitem{JS03} Jacod, J. and  Shiryaev, A. (2003).
 \textit{Limit theorems for stochastic processes},
 second edition. Berlin Springer.
 \bibitem{Kall81}  Kallenberg, O. (1981),
 \textit{Splitting at backward times in regenerative sets}. 
Annals of Probability 9, no. 5, 781--799. 
\bibitem{leon02} L\'eon, J. A., Sol\'e, J. L., Utzet, F. and Vives, J. (2002).
 \textit{On L\'evy processes, Malliavin calculus
and market models with jumps}. Finance Stochast. 6 (2), 197-225.
\bibitem{lokka}  L{\o}kka, A.  (2001). {\it Martingale representation and
  functionals of L\'evy processes}, Preprint series in Pure Mathematics,
 University of Oslo, 21.
\bibitem{MS95} Monat, P. and Stricker, C. (1995).
 \textit{F\"ollmer-Schweizer decomposition and mean-variance hedging
 for general claims}, The Annals of Probability, Vol. 23, No.2, 605-628.
\bibitem{schoutens00} Nualart, D. and Schoutens, W. (2001). \textit{Backward stochastic differential equations and Feynman-Kac formula for L\'evy processes, with applications in finance}, Bernoulli 7, 761-776.
\bibitem{pardouxpeng} Pardoux, E. and Peng, S. (1990).
 \textit{Adapted solution of a backward stochastic differential equation},  Systems Control Lett.  14,  no. 1, 55--61.
\bibitem{Pr92} Protter, P., (2004).
 \textit{Stochastic integration and differential equations},
 second edition. Berlin Springer-Verlag.
\bibitem{Raible98} Raible,
 S. (2000). \textit{L\'evy processes in finance: theory,
 numerics and empirical facts}, PhD thesis,
Freiburg University.
\bibitem{RY} Revuz, D. and Yor, M. (2005).
 \textit{Continuous martingales and Brownian motion}, 
Springer 3rd edition Vol 293.
\bibitem{S94} Schweizer, M. (1994). \textit{Approximating random variables by stochastic integrals}, The Annals of Probability Vol. 22, 1536-1575.
\bibitem{S95} Schweizer, M. (1995). \textit{On the minimal martingale
    measure and the F\"ollmer-Schweizer decomposition}, 
Stochastic Analysis and Applications, 573-599.
\bibitem{S95bis} Schweizer, M. (1995).
 \textit{Variance-optimal hedging in discrete time}, 
Mathematics of Operations Research 20, 1-32.
\end{thebibliography}
\end{document}